\documentclass{amsart}
\usepackage{amsmath}%
\usepackage{amsthm}
\usepackage{amsfonts}%
\usepackage{amssymb}%
\usepackage{graphicx}
\newtheorem{theorem}{Theorem}[section]

\newtheorem*{claim}{Claim}

\newtheorem{corollary}[theorem]{Corollary}

\newtheorem{definition}[theorem]{Definition}

\newtheorem{lemma}[theorem]{Lemma}

\newtheorem{proposition}[theorem]{Proposition}
\newtheorem{remark}[theorem]{Remark}

\numberwithin{equation}{section}
\def\Xint#1{\mathchoice
{\XXint\displaystyle\textstyle{#1}}%
{\XXint\textstyle\scriptstyle{#1}}%
{\XXint\scriptstyle\scriptscriptstyle{#1}}%
{\XXint\scriptscriptstyle\scriptscriptstyle{#1}}%
\!\int}
\def\XXint#1#2#3{{\setbox0=\hbox{$#1{#2#3}{\int}$}
\vcenter{\hbox{$#2#3$}}\kern-.5\wd0}}

\def\dashint{\Xint-}

\newcommand{\R}{\mathbb{R}}

\newcommand{\N}{\mathbb{N}}
\newcommand{\Z}{\mathbb{Z}}

\newcommand{\M}{\mathcal{M}}
\newcommand{\F}{\mathcal{F}}

\newcommand{\Ne}[2]{N^{1,#1}(#2)}

\newcommand{\Nem}[3]{N^{1,#1}(#2;#3)}
\newcommand{\Neml}[3]{N^{1,#1}_{loc}(#2;#3)}

\newcommand{\diaco}[1]{\widehat{#1}_{\operatorname{diag}}}

\newcommand{\spt}{\operatorname{spt}}
\newcommand{\Mod}{\operatorname{Mod}}
\newcommand{\dist}{\operatorname{dist}}
\newcommand{\diam}{\operatorname{diam}}
\newcommand{\diag}{\operatorname{diag}}

\newcommand{\ud}{\mathrm {d}}

\newcommand{\Cp}{\operatorname{Cap}}

\title[Energy minimizers and lifts]{Existence of $p$-energy minimizers in homotopy classes and lifts of Newtonian maps}
\thanks{The author was supported by the Academy of Finland, project no. 1252293, and the V\"ais\"al\"a Foundation.\newline 
\noindent \emph{Email}: elefterios.soultanis@helsinki.fi}
\author{Elefterios Soultanis}
\date{\today}

\address{P.O. Box 68 (Gustaf H\"allstr\"omin katu 2b),\newline
\indent FI-00014 University of Helsinki}
\email{elefterios.soultanis@helsinki.fi}%
\urladdr{http://wiki.helsinki.fi/display/mathstatHenkilokunta/Soultanis\%2C+Elefterios}

\begin{document}
\begin{abstract}
We study the notion of $p$-quasihomotopy in Newtonian classes of mappings and link it to questions concerning lifts of Newtonian maps, under the assumption that the target space is nonpositively curved. Using this connection we prove that every $p$-quasihomotopy class of Newtonian maps contains a minimizer of the $p$-energy if the target has hyperbolic fundamental group.
\end{abstract}
\maketitle

\section{Introduction}

\subsubsection*{Background and setting}
Harmonic maps into manifolds of nonpositive curvature have been studied for at least fifty years. In the seminal paper \cite{eel64} of J. Eells and J. Sampson the authors considered minimizers of the energy $$ E(f)=\frac{1}{2}\int_M|df|^2\ud vol_M$$ for maps $f:M\to N$ between compact manifolds $M$,$N$.

Further work in the seventies, in the setting of maps between manifolds, includes that of R. Schoen and S. Yau \cite{sch76, sch79} which are strongly influenced by topological and algebraic questions related to manifolds. Coming into the nineties the theory found applications in some rigidity questions, see e.g. the paper by M. Gromov  and R. Schoen  \cite{gro92}. N. Korevaar and R. Schoen \cite{kor93} refined the approach in \cite{gro92} and considered energies
\begin{align*}
E^p(f)=&\limsup_{\varepsilon\to 0}\sup_{\varphi \in C_0(\Omega:[0,1])}\int_M\varphi e_\varepsilon(f)\ud vol_M,
\end{align*}
where
\begin{align*}
e_\varepsilon(f)(p)=&\dashint_{B(p,\varepsilon)}\frac{d_Y^p(f(q),f(p))}{\varepsilon^p}\ud vol_M(q)
\end{align*}
for maps $f:\Omega\subset M\to Y$ between a domain of a Riemannian manifold and a \emph{singular} target space. They proved the existence of the limit as  $\varepsilon$ tends to zero in the sense of weak convergence of measures, when $1\le p<\infty$.

For general $p\in (1,\infty)$ in the manifold setting, B. White \cite{whi88} and R. Hardt - F-H.Lin \cite{har87} developed the theory, notably removing the curvature constraints on the target.

Generalizations both towards $p$-harmonic -- rather than harmonic -- maps, and towards more general, nonsmooth, target spaces were natural and necessary steps in light of the applications and scope of the theory. 

A central problem in the setting with nonpositively curved target is to deform a given finite energy map $f$ to a map of minimal energy, that is, a $p$-harmonic representative of the homotopy class of $f$.

This raises several issues. One is the question of defining the energy for maps into nonsmooth spaces. Another, more serious, is the lack of a notion of homotopy for noncontinuous maps.

\bigskip\noindent In the present paper we adopt the framework of Newtonian spaces $\Nem pXY$ together with the notion of $p$-quasihomotopy (see \cite{teri}).

Two $p$-quasicontinuous maps $u,v:X\to Y$ from a metric measure space $X$ to a topological space $Y$ are said to be $p$-quasihomotopic if there is a map $$h:X\times [0,1]\to Y$$ so that given any $\varepsilon >0$ there exists an open set $E\subset X$ with $\Cp_p(E)<\varepsilon$ so that $h|_{X\setminus E\times [0,1]}$ is a classical (continuous) homotopy between $u|_{X\setminus E}$ and $v|_{X\setminus E}$.

The domain $X$ is assumed, throughout the paper, to be a compact doubling $p$-Poincar\'e space. On the target side we assume that $Y$ is a geodesic \emph{path representable space}. The definition is given in Section 3, the gist of it being that, for CW-complexes, path representable spaces are exactly the Eilenberg-Mclane spaces of type $K(\pi,1)$.

The main result of this paper is the following compactness theorem for $p$-qua\-si\-ho\-mo\-to\-py classes.

\begin{theorem}\label{minex}
Let $(X,d,\mu)$ be a compact doubling $p$-Poincar\'e space and $Y$ a compact geodesic path representable space whose fundamental group is Noetherian or torsion free hyperbolic. Then in any $p$-quasihomotopy class $[v]\subset \Nem pXY$ there is a minimizer of the $p$-energy $$e_p(u):=\int_Xg_u^p\ud\mu.$$
\end{theorem}

\noindent In particular if $Y$ is a compact space of nonpositive curvature whose universal cover is Gromov hyperbolic then any $p$-quasihomotopy class in $\Nem pXY$ contains a minimizer of the $p$-energy $e_p$ (Corollary \ref{cat}).

In contrast, the Noetherian property is rather rare in the presence of nonpositive curvature on $Y$. In fact polycyclic groups -- a large subclass of Noetherian groups -- are semihyperbolic (e.g. fundamental groups of CAT(0) or Gromov hyperbolic -spaces) only if they are virtually abelian.

\bigskip\noindent The classical Eells-Sampson result \cite{eel64} is concerned with homotopy classes of continuously differentiable maps between Riemannian manifolds where the target has nonpositive curvature. The proof is based on a gradient flow method of the energy functional and produces a harmonic map homotopic to a given finite energy map. (It would be interesting to know whether, in our generality, two $p$-quasihomotopic \emph{continuous} maps are classically homotopic.) S. Pigola and S. Veronelli \cite{pig09} have studied $p$-harmonic maps and homotopy classes of continuous maps with finite $p$-energy into nonpositively curved manifolds. Their methods incorporate some group theory and also the connection between homotopy and lifts, which is strongly present in this work.

Notable advance in the direction of maps between nonsmooth spaces was made by Korevaar-Schoen \cite{kor93} where the authors define Sobolev spaces and an energy functional for maps from a Riemannian manifold to a metric space. They go on to solve the classical Dirichlet problem for CAT(0) targets, and generalize the Eells-Sampson result to the case of locally CAT(0) target spaces. On the other hand, the Korevaar-Schoen energy functional does not exist (without passing to a subsequence) in the generality we consider.

J. Jost \cite{jost96} further relaxed the assumptions on the domain side to include complete doubling 2-Poincar\'e spaces arising as quotients of the action of a group $\Gamma$ of measure preserving maps on a space. He solved a related problem of finding a $\rho$-equivariant harmonic map for a given homomorphism $\rho$ from $\Gamma$ to the isometry group of the universal cover of the target space (which is assumed to be of nonpositive curvature).

When $Y$ has nonpositive curvature, the homotopy class of a continuous map $u:X\to Y$ corresponds to the class of $\rho$-equivariant maps $\widetilde X\to \widetilde Y$ for $\rho=u_\sharp:\pi(X)\to \pi(Y)$. In our level of generality this connection fails however, since a Newtonian map need not induce a homomorphism between the fundamental groups.

Without this link it is difficult to transfer any topological properties of the domain side to the target. The current work focuses on exploiting geometric group theoretic properties of the fundamental group of the target to obtain existence of $p$-energy minimizers without topological assumptions on the domain side.

\subsection{Tools}

\noindent The proof of Theorem \ref{minex} uses several tools and results, some of which may be of independent interest. Next we describe some of them.

For a topological Hausdorff space $Y$ admitting a universal covering we consider the \emph{diagonal cover of $Y$}, $p:\diaco Y\to Y\times Y$, with the property that $p_\sharp\pi(\diaco Y)\le \pi(Y\times Y)$ is conjugate to $\diag\pi(Y)$. The construction of $\diaco Y$ is given in Section 2.4. (The name is slightly misleading since $\diaco Y$ is actually a covering of $Y\times Y$ and not of $Y$.)

\bigskip\noindent The notion of path representability is defined using the diagonal cover. Suppose $Y$ is a Hausdorff space admitting a universal cover. We say that $Y$ is \emph{path representable} if the component maps of $p=(p_0,p_1):\diaco Y\to Y\times Y$ are homotopic as maps $\diaco Y\to Y$. The name, along with properties of such spaces, will be discussed more in Section 3.

When $Y$ is a length space there is a unique length metric $\hat d$ on $\diaco Y$ making $p$ a local isometry. In this paper we always consider the diagonal cover of $Y$ with this metric when $Y$ is a length space.

\bigskip\noindent The following theorem shows the usefulness of the class of path representable spaces.

\begin{theorem}\label{homchar}
Let $X$ be a compact doubling $p$-Poincar\'e space, $1<p<\infty$, and $Y$ a complete separable geodesic path representable space. Then two maps $u,v\in \Nem pXY$ are $p$-quasihomotopic if and only if there exists a lift $h\in \Nem pX{\diaco Y}$ of the map $(u,v)\in \Nem pX{Y\times Y}$, i.e., a map $h\in \Nem pX{\diaco Y}$ such that $p\circ h=(u,v)$ as maps in $\Nem pX{Y\times Y}$.
\end{theorem}

\noindent Shifting focus to lifts of Newtonian maps, we have the following characterization of the existence of lifts to a given Newtonian map with prescribed value at a given point, in the spirit of the classical version (see for instance \cite{mas67}). Since however Newtonian maps need not induce a homomorphism between the fundamental groups we have to replace the notion of $u_*\pi(X,x_0)$ by a corresponding notion suitable for Newtonian maps.


If $u\in \Nem pXY$ is a map with $p$-weak upper gradient $g\in L^p(\mu)$ and $\Gamma_0$ is such that $g$ is an upper gradient of $u$ along any curve $\gamma$ (and subcurve $\gamma'$) which does not belong to $\Gamma_0$ we set \[ u_\sharp\F_{x_0}(g;\Gamma_0):=\langle [u\circ\gamma]: \gamma\in \F_{x_0}(g;\Gamma_0)\rangle \le \pi(Y,u(x_0)), \] where $\F_{x_0}(g;\Gamma_0)$ is a \emph{fundamental system of loops for $u$} (with basepoint $x_0\in X$), chosen suitably so that it is ``big'' but includes only curves along which $u$ is continuous. Note that $\F_{x_0}(g;\Gamma_0)$ does \emph{not} consist of homotopy classes of paths but rather paths themselves.

The subgroup $u_\sharp\F_{x_0}(g;\Gamma_0)$ is called the \emph{induced subgroup of the fundamental system of loops by $u$} with basepoint $x_0$. The precise definitions are given in Section 4.

\begin{theorem}\label{liftchar}
Suppose $X$ is a compact doubling $p$-Poincar\'e space and $Y$ a complete separable length space with a covering map $\phi:\widehat Y\to Y$ that is a local isometry. Let $u\in \Nem pXY$. There is a null set $N\subset X$ so that for $x_0\in X\setminus N$ and $\hat y_0\in \phi^{-1}(u(x_0))$, the map $u$ admits a lift $h\in \Nem pX{\widehat Y}$ with Lebesgue point $x_0$ and value $h(x_0)=\hat y_0$ in the strong sense 
if and only if
\[ u_\sharp\F_{x_0}(g;\Gamma_0)\le \phi_\sharp\pi(\widehat Y,\hat y_0) \] for some fundamental system of loops $\F_{x_0}(g;\Gamma_0)$ for $u$.
\end{theorem}
Such a lift is necessarily unique (Proposition \ref{uniqlift}). For the notion of a value in the strong sense see Section 2.2


\bigskip\noindent Stability results play an important role in the proof of Theorem \ref{minex}. Having reduced $p$-quasihomotopy to questions of existence of lifts, we seek stability results for sequences of maps admitting a lift to a given covering space. Without any topological assumptions on the domain space $X$ we have to impose further properties on the target side. Specifically the following subconjugacy property of the fundamental group of the covering space (thought of as a subgroup of the fundamental group of the target space) is needed.

For a group $G$ and a subgroup $H$ of $G$ we say that $H$ has the subconjugacy property (with respect to $G$) if, for any sequence $(g_n)\subset G$ there exists some $g\in G$ so that \[ \bigcup_{k\ge 1}\bigcap_{n\ge k}g_n^{-1}Hg_n\le g^{-1}Hg. \] This property is revisited in Section 5.1.

\begin{theorem}\label{stab}
Let $X$ be a compact doubling $p$-Poincar\'e space and $Y$ a complete separable length space with a covering map $\phi:\widehat Y\to Y$ that is a local isometry. Assume $\pi(Y)$ is countable and $\phi_\sharp \pi(\widehat Y)\le \pi(Y)$ satisfies the subconjugacy condition. Suppose $(u_j)$ is a sequence in $\Nem pXY$, each $u_j$ admitting a lift in $\Nem pX{\widehat Y}$, with \[ \sup_j \int_Xg_{u_j}^p\ud\mu<\infty, \] converging in $L^p(X;Y)$ to a map $u\in \Nem pXY$. Then $u$ admits a lift in $\Nem pX{\widehat Y}$.
\end{theorem}

\subsection{Outline of approach}

\subsubsection*{Section 2: Preliminaries} After the introduction, we proceed to present most of the preliminary results needed in the sequel, including basics of Newtonian spaces and Poincar\'e inequalities, some geometric group theory and the \emph{diagonal cover} $\diaco Y$ of a space $Y$. This concept is used in the definition of path representability and plays an important role in connection with $p$-quasihomotopy.

\subsubsection*{Section 3: Path representability} In this section we define the notion of path representable spaces which, for countable CW-complexes, coincides with the notion of $K(\pi,1)$-spaces (Proposition \ref{cw}). The importance of path representable spaces lies in the fact that they satisfy the following universal property.

Let $Y$ be path representable. Then for any topological space $X$, two continuous maps $f,g:X\to Y$ are homotopic if and only if the map $(f,g):X\to Y\times Y$ has a lift $h:X\to \diaco Y$ (Lemma \ref{universal}). Note that Theorem \ref{homchar} is an extension of this universal property to the realm of Newtonian maps.

The section contains some properties of path representable spaces (notably Propositions \ref{kg1} and \ref{cw}) and closes with the proof of Theorem \ref{homchar}.


\subsubsection*{Section 4: Lifts} Theorem \ref{homchar} brings \emph{lifts} into the picture by reducing the question of the existence of a $p$-quasihomotopy to the question of the existence of a lift.

Given a covering map $\phi:\widehat Y\to Y$ that is a local isometry, and a metric measure space $X$, we say that a map $u\in \Nem pXY$ \emph{admits a lift} $h\in \Nem pX{\widehat Y}$ if $\phi\circ h= u$ as maps in $\Nem pXY$.

To characterize existence of lifts we develop an analogue of the classical characterization for continuous maps (see \cite[p. 156, Theorem 5.1]{mas67}). The notion of a \emph{fundamental system of loops} replaces the fundamental group of the domain, and the \emph{induced subgroup of the fundamental system of loops by $u$} replaces the image of the fundamental group under the induced homomorphism of $u$. Note that these definitions do not require a Newtonian map $u\in \Nem pXY$ to induce a homomorphism $\pi(X)\to \pi(Y)$.

The characterization for existence of lifts (Theorem \ref{liftchar}) along with uniqueness of lifts (Proposition \ref{uniqlift}) are proven in Section 4, along with developing the tools for these proofs.

\subsubsection*{Section 5: Stability} The main task in finding energy minimizers in a given $p$-quasihomotopy class is to obtain suitable compactness -- or stability -- properties for these classes. In light of Theorem \ref{homchar}, the main result, Theorem \ref{minex}, can be deduced from Theorem \ref{stab}, a corresponding stability result of a more technical nature for maps admitting a lift to a given covering space. The proof of Theorem \ref{stab} is given in this section.

The technical property, the subconjugacy condition (Definition \ref{subconj}) and some results regarding it are given in the first subsection.

The second subsection contains the proof of Theorem \ref{stab}. After the technical point of passing to slightly lower integrability (Proposition \ref{stabi}), the proof applies Theorem \ref{liftchar}, first to pass from lifts to subgroups. Using these, a suitable limit subgroup is defined and the subconjugacy assumption guarantees that this subgroup satisfies the condition in Theorem \ref{liftchar}, enabling one to pass from the subgroup back to a lift.

The final subsection contains the proof of Theorem \ref{minex}, which uses Theorem \ref{homchar} to reduce the question to showing the existence of a lift, then Theorem \ref{stab} together with Lemmas \ref{noeth} and \ref{hyptor} to show that the energy-minimizing limit map has a lift and Theorem \ref{homchar} again to conclude that it is in the required $p$-quasihomotopy class.

\bigskip\noindent\subsubsection*{Final remark} The methods used here are necessarily quite different than those used in the proofs of the classical analogues of our results even though the general strategy of approaching the problem is certainly inspired by classical theory. The use of geometric group theory, for instance, rids one of the problems of trying to define a homomorphism between fundamental groups induced by a Newtonian map. Incidentally this means that essentially no topological conditions on the domain are required. Naturally this leads to more conditions on the target side but the trade-off is surprisingly favourable. It is also noteworthy that the assumptions of the target side are more of global, rather than local, nature.

\section{Preliminaries}
\subsection{Poincar\'e inequalities}

Throughout this paper we assume that $(X,d,\mu)$ is a metric measure space; that is, a space $X$ with metric $d$ and a \emph{Borel regular} measure $\mu$.
From Section 3 on we assume, in addition, that $(X,d,\mu)$ is a compact doubling metric space supporting a weak $(1,p)$-Poincar\'e inequality, $1 < p<\infty$. We refer to such spaces as (compact, doubling) $p$-Poincar\'e spaces. Doubling here means that the measure $\mu$ is doubling, i.e. that there exists a constant $C<\infty$ so that for all $B(x,r)\subset X$ \[ \mu(B(x,2r))\le C\mu(B(x,r)). \] We always assume $\mu$ to be Borel regular and finite and nonzero on (nontrivial) balls.

The family of all rectifiable curves in a metric measure space $(X,d,\mu)$ is denoted $\Gamma(X)$. Given a family $\Gamma\subset\Gamma(X)$ the \emph{$p$-modulus of $\Gamma$} is defined as \[ \Mod_p(\Gamma;\mu)=\inf\left\{ \int_X\rho^p\ud\mu: 0\le \rho \textrm{ Borel, }\int_\gamma\rho\ge 1\textrm{ for all }\gamma\in \Gamma \right\}. \] When the underlying measure is clear from the context we omit it in the notation and write $\Mod_p(\Gamma)=\Mod_p(\Gamma;\mu)$.

Let $X$ and $Z$ be two metric spaces. A nonnegative Borel function $g:X\to [0,\infty]$ is said to be an \emph{upper gradient} of a map $u:X\to Z$ if, for any rectifiable curve $\gamma:[0,1]\to X$ the inequality 
\begin{equation}\label{ug}
d_Z(u(\gamma(0)),u(\gamma(1)))\le \int_\gamma g
\end{equation} is satisfied. If there is a path family $\Gamma_0$ of zero $p$-modulus so that the estimate (\ref{ug}) holds for all $\gamma\in \Gamma(X)\setminus\Gamma_0$ we say that $g$ is a $p$-weak upper gradient of $u$. If 
\begin{equation} \label{uga}
d_Z(u(\gamma(b)),u(\gamma(a)))\le \int_{\gamma|_{[a,b]}} g,\quad a,b\in [0,1]
\end{equation}
holds for some path $\gamma$ and all $a,b\in [0,1]$ and in addition $\int_\gamma g<\infty$ we say that $g$ is \emph{an upper gradient of $u$ along $\gamma$}. In this case $u\circ\gamma$ is absolutely continuous and  $$ \ell(u\circ\gamma)\le \int_\gamma g.$$ See \cite{haj03} for the definition of line-integrals along rectifiable curves.

We say that a metric space supports a weak $(q,p)$-Poincar\'e inequality, $1\le p,q<\infty$, if there exist constants $C,\sigma$ such that for any pair $(u,g)$ of a locally integrable function $u:X\to \R$ and its locally integrable upper gradient $g:X\to \R$, and a ball $B=B(x,r)\subset X$ of radius $r$, we have \[ \left(\dashint_B|u-u_B|^q\ud\mu\right)^{1/q} \le Cr\left(\dashint_{\sigma B}g^p\ud\mu\right)^{1/p}. \] Here $\sigma B$ is the concentric ball of radius $\sigma r$.

\bigskip\noindent A deep result of Keith and Zhong \cite{zho04} states that, for a complete space, admitting a Poincar\'e inequality is an open ended condition.

\begin{theorem}\label{zho}
Let $(X,d,\mu)$ be a doubling metric measure space such that $X$ is complete. If $X$ supports a weak $(1,p)$-Poincar\'e inequality for some $p>1$ then it supports a weak $(1,q)$-Poincar\'e inequality for some $q<p$ with constants depending only on the data.
\end{theorem}
\begin{proposition}
A compact doubling $p$-Poincar\'e space ($p>1$) supports a weak $(p,q)$-Poincar\'e inequality for some $1\le q<p$.
\begin{proof} This follows from Theorem \ref{zho} and \cite[Theorem 9.1.32]{HKST07}. \end{proof}
\end{proposition}
The following result, due to \cite[Theorem 2]{kei03} will prove helpful when studying images of loops under Newtonian maps (loops based on a fixed point have $p$-modulus zero, yet we want to use them to characterize liftability of Newtonian maps).

Denote by $\Gamma_{xy}$ the set of all rectifiable curves joining $x$ and $y$.
\begin{proposition}
Let $(X,d,\mu)$ be a complete metric measure space with a doubling measure $\mu$, supporting a weak $(1,p)$ - Poincar\' e inequality. Then for each distinct $x,y\in X$
\begin{equation}\label{kei}
d(x,y)^{1-p}\le C \Mod_p(\Gamma_{xy};\mu_{xy}).
\end{equation}
Here \[ \mu_{xy}(A)=\int_A\left[\frac{d(x,z)}{\mu(B(x,d(x,z)))}+\frac{d(y,z)}{\mu(B(y,d(y,z)))} \right]\ud\mu(z). \]
\end{proposition}
For future reference we also record the following facts. Here $\M f$ denotes the maximal function, with respect to balls, of a function $f\in L^1_{loc}(X)$
\begin{lemma}\label{facts}Suppose $X$ is compact and $\mu$ is a doubling measure. If $\rho\ge 0$ is a Borel function with $$\M \rho^p(x)+\M \rho^p(y)<\infty$$ then $$\int_X\rho^p\ud\mu_{xy}<\infty.$$ In particular $$\mu_{xy}(X)<\infty.$$ 

\begin{proof}
Denote $r_0=\diam(X)$. For any $x\in X$ we have
\begin{align*}
\int_X\frac{\rho^p(z)d(x,z)}{\mu(B(x,d(x,z))}\ud\mu(z)=& \sum_{k=0}^\infty \int_{B(x,2^{-k+1}r_0)\setminus B(x,2^{-k}r_0)}\frac{\rho^p(z)d(z,x)}{\mu(B(x,d(x,z)))}\ud\mu(z)\\
\le & \sum_{k=0}^\infty \frac{C2^{-k}r_0}{\mu(B(x,2^{-k}r_0))}\int_{B(x,2^{-k}r_0)}\rho^p\ud\mu \le C\M\rho^p(x).
\end{align*}
Thus
\begin{align*}
\int_X\rho^p\ud\mu_{xy} &= \int_X\frac{\rho^p(z)d(x,z)}{\mu(B(x,d(x,z))}\ud\mu(z)+\int_X\frac{\rho^p(z)d(y,z)}{\mu(B(y,d(y,z))}\ud\mu(z)\\
& \le C\M \rho^p(x)+C\M \rho^p(y)
\end{align*}
for $x,y\in X$ and the result follows.
\end{proof}
\end{lemma}

\bigskip\noindent We will also consider Banach space valued mappings. For them it is useful to have the following result \cite[Theorem 8.1.42]{HKST07}
\begin{theorem}
Suppose $(X,d,\mu)$ is a complete doubling metric measure space. Then $X$ supports a weak $(1,p)$-Poincar\'e inequality for $p>1$ if and only if it supports a weak $(1,p)$-Poincar\'e inequality for $V$-valued maps, for any Banach space $V$, i.e. if there are constants $C',\sigma'\in [1,\infty)$ such that for every locally integrable map $u:X\to V$ and every upper gradient $g$ of $u$ the inequality \[ \dashint_B\|u-u_B\|_V\ud\mu \le C' r \left(\dashint_{\sigma' B}g^p\ud\mu\right)^{1/p} \] holds. The constants $ C'$ and $\sigma'$ depend only on $p$, the data of $X$ and the Banach space $V$.
\end{theorem}

\subsection{Newtonian spaces}

\noindent Let us recall some facts about Newtonian spaces. Suppose $X$ is a doubling $p$-Poincar\'e space and $V$ a Banach space. The Newtonian space $\Nem pXV$ consists of equivalence classes of measurable mappings $u:X\to V$ for which there exists a $p$-integrable upper gradient $g$. The equivalence relation in question is slightly different from the usual identification of almost everywhere agreeing maps: $u$ and $v$ are equivalent if $\|u-v\|=0$ outside a set of $p$-capacity zero, i.e. $\|u-v\|=0$ \emph{$p$-quasieverywhere}. For details and further discussion see \cite{hei01, HKST07}. The local Newtonian space $\Neml pXV$ consists of equivalence classes of \emph{locally} $p$-integrable maps $u:X\to V$ which have an upper gradient $g\in L^p_{loc}(\mu)$.

The space $\Nem pXY$ equipped with the norm \[ \|u\|_{1,p}^p= \|u\|_{L^p(X;V)}^p+\inf_g \|g\|_{L^p(\mu)}^p \] (infimum taken over all $p$-integrable upper gradients of $u$) is a Banach space. If $u\in \Nem pXV$ there is a minimal $p$-weak upper gradient, denoted $g_u\in  L^p(\mu)$ with the property that $g_u$ is a $p$-weak upper gradient of $u$ and if $g$ is any other $p$-integrable upper gradient of $u$ then $g_u\le g$ almost everywhere, see \cite{hei01, HKST07, bjo11}.


\bigskip\noindent For a separable metric space $Y$ we consider the isometric Kuratowski embedding $\lambda : Y\hookrightarrow \ell^\infty$. Define \[ \Nem pXY = \{ u\in \Nem pX{\ell^\infty}: u(x)\in \lambda(Y)\textrm{ for $p$-quasievery }x\in X \} \] and similarly \[ \Neml pXY = \{ u\in \Neml pX{\ell^\infty}: u(x)\in \lambda(Y)\textrm{ for $p$-quasievery }x\in X \}.  \]

\begin{lemma}\label{leb}
Suppose $g$ is a $p$-weak upper gradient of $u\in \Neml pXZ$, where $Z$ is a separable complete metric space and $(X,d,\mu)$ is a doubling metric measure space supporting a weak $(1,p)$-Poincar\'e inequality. If $$x_0\notin \{\M g^p=\infty \}$$ then $x_0$ is a strong Lebesgue point of $u$ in the sense that there exists a unique $a(x_0)\in Y$ so that  \[ \lim_{r\to 0}\dashint_{B(x_0,r)}d^p(u,a(x_0))\ud\mu=0. \]
\begin{proof} Clearly there can be at most one point $a(x_0)$ with the property in the claim. To prove there exists one we consider $Z$ as a closed subset of $\ell^\infty$. Let $0<r\le R$ and $k_0$ the largest integer with the property $r\le 2^{-k}R$. Denoting $B_k=B(x_0,2^{-k}R)$ we have \[ |u_{B(x_0,R)}-u_{B(x,r)}|\le \sum_{k=1}^{k_0}|u_{B_k}-u_{B_{k-1}}|+|u_{B_{k_0}}-u_{B(x_0,r)}|. \] Each term in the sum can be estimated by \[ |u_{B_k}-u_{B_{k-1}}| \le C\dashint_{B_{k-1}}|u-u_{B_{k-1}}|\ud\mu\le C2^{-k}R\M g^p(x_0)^{1/p}  \] where $C>0$ depends only on $p$ and the data of $X$. The last term is estimated similarly by \[ |u_{B_{k_0}}-u_{B(x_0,r)}|\le C2^{-{k_0}}R\M g^p(x_0)^{1/p}. \] Therefore we have the estimate \[ |u_{B(x_0,R)}-u_{B(x_0,r)}|\le \sum_{k=0}^{k_0}C2^{-k}R\M g^p(x_0)^{1/p}=CR\M g^p(x_0)^{1/p}. \] Thus $(u_{B(x_0,R)})_{R>0}\subset \ell^\infty$ is a Cauchy net and converges to a point $a(x_0)\in \ell^{\infty}$.

Now we have the estimate
\begin{align*}
|u-a(x_0)|&\le |u-u_{B(x_0,R)}|+\sum_{k=1}^\infty |u_{B_k}-u_{B_{k-1}}|\\
&\le |u-u_{B(x_0,R)}|+C\sum_{k=1}^\infty 2^{1-k}R\M g^p(x_0)^{1/p}\\
&=|u-u_{B(x_0,R)}|+CR\M g^p(x_0)^{1/p}.
\end{align*}
Thus 
\begin{align}\label{ineq}
&\lim_{R\to 0}\left(\dashint_{B(x_0,R)}|u-a(x_0)|^p\ud\mu\right)^{1/p} \\
\le &\lim_{R\to 0}\left[CR\left(\dashint_{B(x_0,\sigma R)}g_ u^p\ud\mu\right)^{1/p}+ CR \M g^p(x_0)^{1/p}\right]=0.\nonumber
\end{align}

To see that $a(x)\in Z$ we note that $$ \dist_{\ell^\infty}(Z,a(x_0))\le |u(x)-a(x_0)|$$ for $p$-quasievery $x\in X$. Applying the inequality (\ref{ineq}) it follows that $$\dist_{\ell^\infty}(Z,a(x_0))=0$$ and the claim follows.
\end{proof}
\end{lemma}

\noindent We say that $a(x_0)$ is the strong value of $u$ at $x_0$. From now on it is understood without further mention that the value $f(x_0)$ of a Newtonian mapping $f$ at a Lebesgue point $x_0$ is the strong value.

\begin{theorem}[Rellich-Kondrachov]\label{rellich}
Let $X$ be a doubling metric measure space supporting a weak $(1,p)-$ Poincar\'e inequality, and $Y$ a proper metric space. If $(u_j)$ is a sequence in $\Neml pXY$, and $v\in \Neml pXY$ with \[ \sup_{j}\left[\int_Bd_Y^p(v,u_j)\ud\mu+\int_{5\sigma B}g_{u_j}^p\ud\mu\right] <\infty, \] for a given ball $B\subset X$, then there is a subsequence (denoted by the same indices) and $u\in L^p(B;Y)$ so that $$\|u_j-u\|_{L^p(B;Y)}\to 0$$ as $j\to \infty$ and, moreover, $$\int_{ B}g_u^p\ud\mu\le \liminf_{j\to \infty}\int_{ B}g_{u_j}^p\ud\mu.$$
\end{theorem}
For a proof see \cite[Theorem 8.3]{haj00}, \cite[Theorem 1.3]{kor93} and \cite[Theorem 1.11]{teri}.

\subsection{Some geometric group theory}

\noindent Here we collect some basic concepts in geometric group theory. We adopt the conventions and terminology from \cite{bri99}, particularly Chapter I.8.
\begin{definition}
A group $G$ is said to satisfy the ACC (Ascending Chain Condition) for a family $\mathcal G$ of subgroups of $G$ if, given any ascending chain $H_1\le H_2\le \ldots \le G$ in $\mathcal G$ stabilizes, i.e. there exists a number $n_0\in \N$ so that $H_n=H_{n_0}$ for all $n\ge n_0$.
\end{definition}

\begin{definition}
A group $G$ is said to be Noetherian if it satisfies the ACC for all subgroups.
\end{definition}
\begin{remark}
A group is Noetherian if and only if any subgroup $H$ of $G$ is finitely generated (including $G$ itself).
\end{remark}

\begin{definition}\label{hyperbolic}
A finitely generated group $G$ is said to be \emph{hyperbolic} if the Cayley graph $\mathcal C_S(G)$ with respect to some finite generating set equipped with the word metric is a Gromov-hyperbolic space.
\end{definition}
One can equivalently demand the condition with respect to \emph{all} finite generating sets. See \cite[Chapter I.8]{bri99} and \cite[Chapter III.H]{bri99} for definitions and basic facts of Cayley graphs and Gromov-hyperbolic spaces.

\begin{proposition}\label{svarc}(Svarc-Milnor,\cite[Chapter I.8, Proposition 8.19]{bri99})
Let $X$ be a length space. If a group $G$ acts properly and cocompactly by isometries on $X$ then $G$ is finitely generated and for any $x_0\in X$ the map $g\mapsto g.x_0: G\to X$ is a quasi-isometry.
\end{proposition}
It is a consequence of the Svarc-Milnor lemma and the quasi-isometry invariance of Gromov hyperbolicity that a finitely generated group $G$ is hyperbolic if and only if it acts properly and cocompactly by isometries on a Gromov-hyperbolic space.

\begin{proposition}\cite[Chapter II.4, Theorem 4.13]{bri99}\label{torfree}
Let $G=\pi(Y)$ be the fundamental group of a complete nonpositively curved space $Y$. Then $G$ is torsion free.
\end{proposition}

\begin{proposition}\cite[Chapter III.$\Gamma$, Corollary 3.10]{bri99}\label{zed}
Let $G$ be a finitely generated hyperbolic group. Then for any element $g\in G$ of infinite order the cyclic subgroup $\langle g\rangle$ generated by $g$ has finite index in the centralizer $C(g)$.
\end{proposition}

\begin{proposition}\cite[Chapter II.7, Theorem 7.5]{bri99}
Let $G$ be a finitely generated group acting properly and cocompactly by isometries on a CAT(0) space $Y$. Then $G$ satisfies the ACC for virtually abelian subgroups.
\end{proposition}
We also have an analogue for hyperbolic groups.
\begin{lemma}\label{ACC}
A hyperbolic group satisfies the ACC for virtually abelian subgroups.
\begin{proof}
The proof of the abovementioned Theorem 7.5 in \cite{bri99} gives this result as well (see the remarks after the proof therein).
\end{proof}
\end{lemma}

\subsection{The diagonal cover} Recall the definition of the compact-open topology from \cite[Definition 43.1]{wil04}: given topological spaces $X$ and $Y$, the \emph{compact-open} topology on the set $C(X;Y)$ of continuous maps $X\to Y$ is the topology induced by the subbasis consisting of sets of the form \[ \{ f\in C(X;Y): f(K)\subset U \} \] for compact $K\subset X$ and open $U\subset Y$. When $X$ is compact and $Y$ a metric space the compact-open topology coincides with the topology of uniform convergence, \cite[Theorem 43.7]{wil04}.

\begin{definition}\label{diacover}
Let $Y$ be a connected, locally path connected Hausdorff  space admitting a universal cover. Consider the space $C([0,1];Y)$ with the compact-open topology, and the equivalence relation $\sim$, where $\alpha\sim\beta$ if $\alpha$ and $\beta$ are endpoint preserving homotopic as paths. We call the set \[ \diaco Y=C([0,1];Y)/\sim \] with the quotient topology the \emph{diagonal cover} of $Y$.
\end{definition}

The name is indicative of the fact that the map $p:\diaco Y\to Y\times Y$ given by
\begin{align}\label{diaco}
p([\gamma])=(\gamma(0),\gamma(1))
\end{align}
is a covering map and $p_\sharp\pi(\diaco Y)=\diag(\pi(Y))\le \pi(Y)\times \pi(Y)$ (Lemma \ref{diagonal}).

In case $Y$ is a length space there is a unique metric $\hat d$ on $\diaco Y$ making it a length space and $p$ a local isometry \cite[Proposition 3.5.1]{pap05} and we always consider $\diaco Y$ with this metric.

The construction of the diagonal cover is similar to a construction of a universal cover for $Y$ (the difference being that there one chooses an arbitrary point and looks at all paths starting from that point), and in a similar way we may see that the projection map (\ref{diaco}) is a well-defined covering map. See \cite{bro12} for a discussion on different constructions for, and generalizations of, a universal covering spaces.



It is easily seen that $\widetilde Y_x=\{ [\gamma]\in \diaco Y|\ \gamma(0)=x \}$ is a closed subspace of $\diaco Y$ and the induced topology on $\widetilde Y_x$ makes it the universal covering of $Y$.

\begin{lemma}\label{diagonal}
Let $y_0\in Y$ be a point in the connected, locally path connected Hausdorff space $Y$, and consider the constant path $\hat y_0\equiv y_0$. We have \[ p_\sharp \pi(\diaco Y,\hat y_0)=\diag(\pi(Y,y_0)) \] where $\diag(\pi(Y,y_0))$ is the diagonal subgroup of $\pi(Y,y_0)\times \pi(Y,y_0)$.
\begin{proof}
Suppose $\hat{\gamma}:[0,1]\to \diaco Y$ is a path in $\diaco Y$ with start- and endpoint $\hat y_0$. Then $(t,s)\mapsto \hat{\gamma}(s)(t)$ is a homotopy relative $\hat y_0$ from $\hat{\gamma}(0)$ to $\hat\gamma(1)$. Since $p\circ\hat{\gamma}=(\hat\gamma(0), \hat\gamma(1))$ we have $p_{\sharp}(\hat\gamma)\in \diag (\pi(Y,y_0))$. This shows that $$p_\sharp\pi(\diaco Y,\hat y_0)\le \diag(\pi(Y,y_0)).$$

For the other inclusion take any path $(\gamma,\gamma)\in \diag(\pi(Y,y_0))$ and define the path  $\hat\gamma$ by setting $$\hat \gamma(t)=\textrm{ constant path }\equiv \gamma(t).$$
\end{proof}
\end{lemma}

\bigskip\noindent Suppose now that $Y$ is, in addition, a separable length space. If $Y$ is also proper then every homotopy class of paths in $Y$ contains a local geodesic. This is proven in \cite[Proposition 2.4.11]{pap05} under the assumption that $Y$ is locally simply connected. However, the proof works under the assumption of semi-local simply connectedness as well.

\section{Path representability}


\begin{definition}\label{prp}
A connected, locally path connected Hausdorff space is called \emph{path representable} if the component maps of $p=(p_0,p_1):\diaco Y \to Y\times Y$ are homotopic as maps $p_0,p_1:\diaco Y \to Y$.
\end{definition}
The idea behind this term is that, somehow, the homotopy classes of paths in $Y$ are ``representable'' by paths (not necessarily in the given homotopy class) which vary continuously with their endpoints.

The class of path representable spaces generalizes the nonpositively curved ones while retaining the universal property below, essential for our approach and the validity of Theorem \ref{homchar}.

\begin{proposition}\label{universal}
The property of a space $Y$ being path representable is equivalent to the following universal property: For every topological space $X$, two continuous maps $f,g:X\to Y$ are homotopic (in the usual sense) if and only if the map $(f,g):X\to Y\times Y$ has a lift $h:X\to \diaco Y$.
\begin{proof}
To see that path representability implies the universal property, note that if $H:X\times [0,1]\to Y$ is a homotopy between two maps $f,g:X\to Y$ we obtain a lift $\hat H:X\to \diaco Y$ of $(f,g)$ by setting $$ \hat H(x)=[H(x,\cdot)].$$ If $h:p_0\simeq p_1$ is a homotopy then the existence of a lift $\hat H:X\to \diaco Y$ of $(f,g):X\to Y\times Y$ yields a homotopy $H:f\simeq g$ by $H(x,t)=h(\hat H(x),t)$. 

\bigskip\noindent Conversely, if $Y$ satisfies the universal property, the identity $id:\diaco Y\to \diaco Y$ is a lift of the map $(p_0,p_1):\diaco Y\to Y\times Y$.
\end{proof}
\end{proposition}

In other words the class of path representable spaces is exactly those spaces where homotopy of two maps is characterized by the existence of a lift of the product map to the diagonal cover. Let us mention some further properties of path representable spaces.

\begin{proposition}\label{kg1}
A path representable space $Y$ is $K(G,1)$ with $G=\pi(Y)$.
\begin{proof}
Let $h:p_1\simeq p_0$ be a homotopy. Fix an arbitrary point $q\in Y$ and consider the subset $\tilde Y_q=\{[\gamma]: \gamma(0)=q\}\subset \hat Y$. Now set $r:\tilde Y_q\times [0,1]\to Y$, \[ r([\gamma],t)=h([\gamma],t). \] Then $r(\cdot,0)=p_1(\cdot)$ and $r(\cdot, 1)=p_0(\cdot)\equiv q$. Since $\tilde Y_q$ is the universal cover of $Y$ it follows that $\tilde Y_q\times [0,1]$ has trivial fundamental group and $r$ has a lift to $\tilde r:\tilde Y_q\times [0,1]\to \tilde Y_q$ satisfying $\tilde r([q],0)=[q]$. This shows that the universal cover is contractible.
\end{proof}
\end{proposition}

In fact for CW-complexes the converse is true as well.

\begin{proposition}\label{cw}
Suppose $Y$ is a $K(G,1)$ path connected CW-complex with countably many cells. Then $Y$ is path representable.
\begin{proof}
The countability of cells implies that $Y\times Y$ (with the product topology) is a CW complex. Since $\diaco Y$ is a covering of $Y\times Y$ it is also a CW-complex (see \cite[Appendix]{hat02}). By \cite[Proposition 1B.9, p. 90]{hat02} the components $p_0$ and $p_1$ of $p=(p_0,p_1):\diaco Y\to Y\times Y$ are homotopic if $(p_0)_\sharp=(p_1)_\sharp$ as homomorphisms $\pi(\diaco Y, \hat y_0)\to \pi(Y\times Y, (y_0,y_0))$. Let $[\alpha]\in \pi(\diaco Y, \hat y_0)$. Then $H:[0,1]^2\to \diaco Y$ defined by $$ H(t,s)=\alpha(t)(s)$$ is a homotopy from $p_0\circ\alpha$ to $p_1\circ\alpha$. Thus $$(p_0)_\sharp([\alpha])=(p_1)_\sharp([\alpha]).$$

\end{proof}
\end{proposition}

\bigskip\noindent Path representability is a purely topological property which is invariant under homotopy equivalence. When working with metric spaces one may ask when are they path representable. Nonpositive curvature guarantees path representability.

In some ways path representability can be thought of as a topological substitute for nonpositive curvature (in any case, it is \emph{the} topological property of nonpositively curved spaces crucial to the results in this section). 

Another sufficient criterion for a metric space to be path representable is given in the next lemma. We say that a metric space $Y$ is uniquely locally geodesic if each pair of points may be connected by a unique geodesic and, moreover, every local geodesic is a (global) geodesic.

\begin{lemma}\label{ulg}
Suppose the universal cover of $Y$ is a proper uniquely locally geodesic space. Then $Y$ is path representable.
\begin{proof}
We prove that each homotopy class $[\gamma]$ of a path $\gamma$ in $Y$ contains a unique local geodesic. By the remark after the definition of $\hat Y$ there is at least one local geodesic in $[\gamma]$. Let $\gamma_1,\gamma_2$ be two local geodesics in $[\gamma]$. Their lifts $\tilde \gamma_1$, $\tilde \gamma_2$ with common starting point are local geodesics, which also have a common endpoint. Therefore the lifts are in fact the same and consequently $\gamma_1=\gamma_2$.

Now we may define a map $h:\hat Y\times [0,1]\to Y$ by $h([\gamma],t)=\gamma'(t)$, where $\gamma'$ is the unique local geodesic in $[\gamma]$. Since in proper uniquely locally geodesic spaces, the local geodesics vary continuously with their endpoints (this follows easily from the Arzela-Ascoli Theorem, see \cite[Definition 2.4.14]{pap05} and the discussion after it), we see that $h$ is a continuous map. Moreover, since $h([\gamma],0)=\gamma(0)=p_0([\gamma])$ and similarly $h(\cdot, 1)=p_1(\cdot)$, we have constructed a homotopy $h: p_0\simeq p_1$.
\end{proof}
\end{lemma}

\begin{corollary}\label{cat}
Let $Y$ be a compact nonpositively curved space whose universal cover is Gromov hyperbolic. Then $Y$ is path representable and $\pi(Y)$ is torsion free hyperbolic. In particular $Y$ satisfies the conditions of Theorem \ref{minex}.
\begin{proof}
The universal cover $\widetilde Y$ of $Y$ is a proper CAT(0) space, whence by Lemma \ref{ulg} $Y$ is path representable. Moreover $\pi(Y)$ is a finitely generated group acting properly and cocompactly by isometries on $\widetilde Y$. By the Svarc-Milnor Lemma \ref{svarc} $\pi(Y)$ is quasi-isometric to $\widetilde Y$ and therefore Gromov hyperbolic. Torsion freedom is given by Proposition \ref{torfree}.
\end{proof}
\end{corollary}


\noindent Let us conclude this section with a proof of Theorem \ref{homchar}.

\begin{proof}[Proof of Theorem \ref{homchar}]
Suppose $u,v\in \Nem pXY$ are $p$-quasihomotopic and let $h:u\simeq v$ be a $p$-quasihomotopy. Define a map $\hat h:X\to \diaco Y$ by \[ \hat h(x)=[h(x,\cdot)] \] as in the continuous case. Then $\hat h$ is $p$-quasicontinuous, since for any $\varepsilon >0$ we may find an open set $U\subset X$ with $\Cp_p(U)<\varepsilon$ so that $h|_{X\setminus U\times [0,1]}$, and consequently $\hat h|_{X\setminus U}=[h|_{X\setminus U\times [0,1]}]$, is continuous.

Further, $$p\circ \hat h (x)=(h(x,0),h(x,1))=(u(x),v(x))$$ outside a set $E$ of zero $p$-capacity. To prove that $\hat h\in \Nem pX{\diaco Y}$ let $\Gamma$ be path family with $\Mod_p(\Gamma)=0$ so that any $\gamma\notin\Gamma$ omits $E$, $g_u$ and $g_v$ are upper gradients for $u$ and $v$ along $\gamma$, and $\hat h\circ \gamma$ is continuous. Then $g:=(g_u^2+g_v^2)^{1/2}$ is an upper gradient of $(u,v)$ along $\gamma$.

By the compactness of $((u,v)\circ \gamma)[0,1]$ we may find $r>0$ so that $B((u,v)\circ\gamma(t),r)$ is a covering neighbourhood for each $t\in [0,1]$.

Partition the unit interval as $0=a_0<\ldots <a_m=1$ so that $$(u,v)\circ\gamma([a_{j},a_{j+1}])\subset B((u,v)\circ\gamma(a_j),r/2).$$ The values of $\hat h$ along $\gamma|_{[a_j,a_{j+1}]}$ vary in one component of $p^{-1}B((u,v)\circ\gamma(a_j),r/2)$, and therefore \[ \hat d(\hat h(\gamma(a_{j+1})),\hat h(\gamma(a_j)))= d_{Y^2}((u,v)(\gamma(a_{j+1})),(u,v)(\gamma(a_j)))\le \int_{\gamma|_{[a_j,a_{j+1}]}}g \] for each $j=0,\ldots , m-1$. Thus
\begin{align*}
\hat d(\hat h(\gamma(1)),\hat h(\gamma(0)))\le \sum_{j=0}^{m-1}\hat d(\hat h(\gamma(a_{j+1})),\hat h(\gamma(a_j)))\le \sum_{j=0}^{m-1}\int_{\gamma|_{[a_j,a_{j+1}]}}g=\int_\gamma g.
\end{align*}
This proves that $\hat h\in \Nem pX{\diaco Y}$.

\bigskip\noindent Now suppose $\hat f\in \Nem pX{\diaco Y}$ is a lift of $(u,v)\in \Nem pX{Y\times Y}$. Set $F(x,t)=h(\hat f(x),t)$ where $h:p_0\simeq p_1$ is a homotopy. Let $\varepsilon >0$ and let $E\subset X$ be an open set with $\Cp_p(E)<\varepsilon$ such that $\hat f|_{X\setminus E}$ is continuous. Then $F|_{X\setminus E\times [0,1]}$ is continuous, and $F(x,0)=h(\hat f(x),0)=u(x)$, $F(x,1)=h(\hat f(x),1)=v(x)$ for $x\in X\setminus E$. Therefore $F$ is a $p$-quasihomotopy $u\simeq v$.
\end{proof}

\section{Lifts}

\noindent In this section we focus on lifts of Newtonian maps. \emph{Throughout this section we assume that $X$ is a compact doubling $p$-Poincar\'e and that $Y$ is a complete separable length space equipped with a covering $\phi:\hat Y\to Y$ that is a local isometry.} 


\bigskip\noindent A covering neighbourhood $U\subset Y$ is a path connected neighbourhood so that the map $\phi|_V:V\to U$ is an isometry onto $U$, for any component $V$ of $\phi^{-1}U$, and the induced homomorphism $\iota_\sharp:\pi_1(U)\to \pi_1(Y)$ of the inclusion $\iota: U\to Y$ is trivial. Thus, if $B(y,r)\subset Y$ is a covering neighbourhood, we have $$ \phi^{-1}B(y,r)=\bigcup_{\hat y\in \phi^{-1}(y)}\widehat B(\hat y, r),$$ where $$\widehat B(a,r)=\{b\in \widehat Y: \hat d(a,b)<r \}.$$ We shall make extensive use of these facts without further mention. 

\begin{definition}\label{fsl}
Let $g\in L^p(\mu)$ be a nonnegative Borel function, $\Gamma_0$ a path family of zero $p$-modulus and $x_0\in X$ a point with $\M g^p(x_0)<\infty$ and $x_0\notin \spt_p\Gamma_0$. The path family given by \[ \F=\F_{x_0}(g;\Gamma_0):=\{\alpha\beta^{-1}: \alpha(1)=\beta(1)=:x, \ \ \alpha,\beta\in \Gamma_{x_0x}\setminus \Gamma_0,\ \ \M g^p(x)<\infty \} \] is called the \emph{fundamental system of loops} with basepoint $x_0\in X$.
\end{definition}
Here $\Gamma_{xy}$ denotes the family of rectifiable curves joining the points $x$ and $y$; the $p$-support of a path family $\Gamma_0$ of zero $p$-modulus is $$\spt_p\Gamma_0:=\bigcap_\rho\{ \M \rho^p=\infty \}$$ where the intersection is taken over all $p$-integrable Borel functions $\rho\ge 0$ such that $$\int_\gamma\rho=\infty \quad \textrm{ for all }\gamma\in\Gamma_0.$$

Let $u\in \Nem pXY$ and let $g\in L^p(\mu)$ be a $p$-weak upper gradient of $u$. Suppose $\Gamma_0$ and $x_0$ in Definition \ref{fsl} are such that $g$ is an upper gradient of $u$ along any loop $\gamma\in \F_{x_0}(g;\Gamma_0)$. Then we say that $\F_{x_0}(g;\Gamma_0)$ is a fundamental system of loops\emph{ for $u$} with basepoint $x_0$.

\begin{definition}\label{imgp}
If $u\in \Nem pXY$ and $\F_{x_0}(g;\Gamma_0)$ is a fundamental system of loops for $u$, with basepoint $x_0$, we may associate to them the subgroup \[ u_\sharp\F_{x_0}(g;\Gamma_0):=\langle [u\circ\gamma]:\gamma\in \F_{x_0}(g;\Gamma_0)\rangle\le \pi(Y,u(x_0)). \] We call this the \emph{induced subgroup of the fundamental system of loops by $u$}.
\end{definition}

Note that in this definition we do \emph{not} require $u$ to induce a homomorphism $\pi(X)\to \pi(Y)$. 

\bigskip\noindent Our first lemma deals with minimal $p$-weak upper  gradients of lifts.

\begin{lemma}\label{liftgrad}
Let $u\in \Nem pXY$ and $h\in \Nem pX{\hat Y}$ be a lift of $u$. Then $g_{h}=g_u$ almost everywhere.
\end{lemma}
\begin{proof}
Since \[ d(u(x),u(y))\le \hat d(h(x),h(y))\le \int_\gamma g_{h} \] for $p$-almost every curve, it follows that $g_u\le g_{h}$ almost everywhere. 

To see the opposite inequality choose a path family $\Gamma$ of $p$-modulus zero such that $g_u$ is an upper gradient of $u$ along $\gamma$ and $\phi\circ h \circ \gamma=u\circ \gamma$ whenever $\gamma\notin \Gamma$.

By the (absolute) continuity of $u\circ \gamma$ we may choose $0=a_0<\cdots <a_n=1$ and points $y_0,\ldots,y_{n-1}\in Y $ so that $u\circ \gamma ([a_j,a_{j+1}])\subset B(y_j,r_j)$, where $B(y_j,r_j)$ is a covering neighbourhood for all $j=0,\ldots , n-1$ and consequently $\phi$ restricted to $h\circ\gamma([a_j,a_{j+1}])$ is an isometry. Thus
\begin{align*}
\hat d(h(\gamma(0)),h(\gamma(1)))&\le \sum_{j=0}^{n-1} \hat d(h(\gamma(a_j)),h(\gamma(a_{j+1}))) =\sum_{j=0}^{n-1} d( u(\gamma(a_j)), u(\gamma(a_{j+1}))) \\
&\le \sum_{j=0}^{n-1} \int_{\gamma|_{[a_j,a_{j+1}]}}g_u=\int_\gamma g_u.
\end{align*}
Thus $g_{h}\le g_u$ almost everywhere.
\end{proof}

\bigskip\noindent A rather direct consequence of Lemma \ref{liftgrad} is that any $p$-integrable $p$-weak upper gradient of $u$ is also a $p$-weak upper gradient of any lift of $u$, and vice versa.

The following lemma is a version of the uniqueness of lifts in the Newtonian setting.
\begin{proposition}\label{uniqlift}
Suppose $h_1,h_2\in \Nem pX{\hat Y}$ are two lifts of a Newtonian map $u\in \Nem pXY$. If $h_1$ agrees with $h_2$ on a set of positive $p$-capacity then they agree $p$-quasieverywhere.
\begin{proof}
Consider the function $f:X\to\R$, given by $x\mapsto \hat d(h_1(x),h_2(x))$. Then $f\in \Ne pX$. In light of \cite[Lemma 4.3]{teri} it suffices to show that $F:=\{f=0\}$ is both $p$-quasiclosed and $p$-quasiopen (see \cite[Definition 4.2]{teri} for the definitions). Quasiclosedness is evident since $f$ is $p$-quasicontinuous.

To see that $F$ is $p$-quasiopen, let $\varepsilon >0$ and $U\subset X$ be an open set for which $\Cp_p(U)<\varepsilon$ and $f|_{X\setminus U}$ is continuous. Then $f(X\setminus U)\subset Y$ is compact and there is a positive radius $r_0>0$ so that $B(x,r_0)$ is a covering neighbourhood of $x$ for all $x\in X\setminus U$.

We conclude that if $x\in X\setminus U$ then $f(x)=\hat d(h_1(x),h_2(x))=0$ if and only if $f(x)< r_0$ since the set of points lying on the fibre $\phi^{-1}(u(x))$ is discrete. Consequently we have the equality $$F\setminus U=\{x\in X\setminus U: f(x)<r_0 \}=(f|_{X\setminus U})^{-1}((-\infty,r_0))$$ so that $F\setminus U$ is open in $X\setminus U$. This finishes the proof.
\end{proof}
\end{proposition}



\bigskip\noindent We set notation for the rest of this section. The set of all rectifiable paths in $X$ is denoted by $\Gamma(X)$. For $A\subset X$ we denote $$\Gamma_A=\{ \gamma\in\Gamma(X): \gamma^{-1}(A)\ne \varnothing \}.$$ For $g\in L^p(\mu)$ let \[ \Gamma_g=\{ \gamma\in \Gamma(X): \int_\gamma g=\infty \} \] and for a pair $(u,g)$ of a map and its $p$-weak upper gradient we let \[\Gamma_{u,g}=\{\gamma\in \Gamma(X): g\textrm{ is an upper gradient of $u$ along }\gamma \}. \] Note that $$\Gamma_g\subset \Gamma(X)\setminus\Gamma_{u,g}$$ and $$\Mod_p(\Gamma(X)\setminus\Gamma_{u,g})=0 $$ if $g$ is a $p$-weak upper gradient of $u$. Moreover $$\Mod_p(\Gamma_F)=0$$ if and only if $\Cp_p(F)=0$.

\begin{lemma}\label{trivial}
Let $\Gamma_0$ be a path family of $p$-modulus zero. Let $\rho \ge 0$ be a $p$-integrable Borel function such that $$\int_\gamma\rho=\infty\textrm{ for all }\gamma\in \Gamma_0$$  and suppose $x,y\notin \{\M \rho^p=\infty\}$. Then \[ \Mod_p(\Gamma_{xy}\cap\Gamma_0;\mu_{xy})=0. \]
\begin{proof}
By Lemma \ref{facts} $$\int_X\rho^p\ud\mu_{xy}<\infty.$$Thus there is a nonnegative Borel map  $\rho \in L^p(\mu_{xy})$ such that $$\int_\gamma\rho=\infty \textrm{ for all }\gamma\in \Gamma_{xy}\cap\Gamma_0.$$ The claim follows.
\end{proof}
\end{lemma}

\begin{lemma}\label{t1}
Suppose $\rho\ge 0$ is a $p$-integrable Borel function and $x_0\in X$ a point such that $\M \rho^p(x_0)<\infty$. Define the function $f_{x_0,\rho}:X\to [0,\infty]$ by \[ x\mapsto\inf_{\gamma\in\Gamma_{x_0x}}\int_\gamma\rho. \] Then $f_{x_0,\rho}$ has $p$-weak upper gradient $\rho$.
\end{lemma}
\begin{proof}
We abbreviate notation by setting $f_{x_0,\rho}=:f$. By Lemma \ref{trivial} $f(x)<\infty$ for any $x\notin \{ \M \rho^p=\infty \}=:N$. Suppose $\gamma$ is a path such that $\gamma(0),\gamma(1)\notin N$. If $\int_\gamma\rho=\infty$ the upper gradient inequality (\ref{ug}) holds trivially. Suppose  $\int_\gamma\rho<\infty$ and let $\alpha\in\Gamma_{x_0\gamma(0)}\setminus\Gamma_\rho$ be a path. Then $\gamma\alpha\in\Gamma_{x_0\gamma(1)}\setminus\Gamma_\rho$, and $$ f(\gamma(1))\le \int_\alpha\rho+\int_\gamma\rho.$$ Taking infimum over $\alpha$ we obtain $$f(\gamma(1))\le f(\gamma(0))+\int_\gamma\rho.$$ Interchanging the roles of $\gamma(0)$ and $\gamma(1)$ we obtain $$|f(\gamma(1))-f(\gamma(0))|\le \int_\gamma\rho.$$ By \cite[Proposition 1.50]{bjo11} it follows that $\rho$ is a $p$-weak upper gradient of $f$.
\end{proof}

\begin{lemma}\label{lebpoint}
Suppose $\rho\ge 0$ is a $p$-integrable Borel function and $x_0\in X$ a point such that $\M \rho^p(x_0)<\infty$. Then the map $f_{x_0,\rho}$ defined in Lemma \ref{t1} has a strong Lebesgue point $x_0$ and the strong value 0: \[ \lim_{r\to 0}\dashint_{B(x_0,r)}f^p\ud\mu=0.\]
\end{lemma}
\begin{proof}
Clearly $$ 0\le f_{x_0,\rho}\le f_{x_0,1+\rho}=:h.$$ It suffices to prove the claim for $h$ which, by Lemma \ref{t1}, has a $p$-weak upper gradient $1+\rho$. First we prove the claim without the exponent $p$.

Let $n\ge 0$ and set $$\rho_n=\min\{n,\rho\}\textrm{ and }h_n=f_{x_0,1+\rho_n}. $$ Since $\rho_n\le n$ and $X$ is quasiconvex it follows that $h_n$ is a $Cn$-Lipschitz function satisfying $h_n\le h$ pointwise.

\bigskip\noindent For each fixed $n$, every point is a Lebesgue point for $h_n$ and the strong value (cf. Lemma \ref{leb}) $a_n(x)$ at any point agrees with the value of the extension of $h_n|_{X\setminus\{ \M \rho^p=\infty \}}$ to a Lipschitz function on the whole $X$. It follows that if $\gamma\in\Gamma_{x_0x}\setminus\Gamma_\rho$ for a point $x\notin \{ \M\rho^p=\infty \}$ then $a_n(x_0)$ is given by $\lim_{t\to 0}h_n(\gamma(t))$. This clearly equals zero. Thus $x_0$ is a Lebesgue point for $h_n$ and the strong value equals 
\begin{equation}\label{denspo}
\lim_{r\to 0}\dashint_{B(x_0,r)}h_n\ud\mu=\lim_{t\to 0}h_n(\gamma(t))=0.
\end{equation}

\bigskip\noindent Let $x\notin \{ \M\rho^p=\infty \}$. We want to prove that $$h(x)\le \liminf_{n\to \infty}h_n(x). $$ To this end fix a sequence $(n_k)$ so that $$\liminf_{n\to\infty}h_n(x)=\lim_{k\to\infty}h_{n_k}(x).$$ For each $k$, let $\gamma_{n_k}\in \Gamma_{x_0x}\setminus\Gamma_\rho$ be so that $$\int_{\gamma_{n_k}}(1+\rho_{n_k})<h_{n_k}(x)+2^{-k}.$$ The sequence $(\gamma_{n_k})$ has length uniformly bounded by \[ \ell(\gamma_{n_k})\le \int_{\gamma_{n_k}}(1+\rho_{n_k})<h_{n_k}(x)+2^{-k}\le h(x)+1. \] By \cite[Proposition 4]{kei03}, $(\gamma_{n_k})$ subconverges to a path $\gamma\in\Gamma_{x_0x}$ and $$\int_\gamma(1+\rho)\le \liminf_{k\to\infty}\int_{\gamma_{n_k}}(1+\rho_{n_k})$$ (incidentally showing that $\gamma\notin\Gamma_\rho$) from which it follows that 
\begin{equation}\label{fatou}
h(x)\le \lim_{k\to\infty}(h_{n_k}(x)+2^{-k})= \liminf_{n\to \infty}h_n(x).
\end{equation}

\bigskip\noindent Using (\ref{denspo}) and \cite[Theorem 8.1.55]{HKST07} we estimate \[ \dashint_{B(x_0,r)}h_n\ud\mu=|h_n(x_0)-(h_n)_{B(x_0,r)}|\le Cr\M(1+\rho_n)^p(x_0)^{1/p}\le Cr\M(1+\rho)^p(x_0)^{1/p} \] for all $n$. Together with (\ref{fatou}) and Fatou's lemma we have \[ \dashint_{B(x_0,r)}h\ud\mu\le Cr\M(1+\rho)^p(x_0)^{1/p}\to 0  \]as $r\to 0$.

\bigskip\noindent Once $$\lim_{r\to 0}\dashint_{B(x_0,r)}h\ud\mu=0$$ is established we use the Poincar\'e inequality to obtain the estimate \[ \left(\dashint_{B(x_0,r)}h^p\ud\mu\right)^{1/p} \le \dashint_{B(x_0,r)}h\ud\mu+Cr\M\rho^p(x_0)^{1/p}.\] From this the claim follows.
\end{proof}

\begin{lemma}\label{t0}
Let $u\in \Nem pXY$ and let $g\in L^p(\mu)$ be a $p$-weak upper gradient of $u$. Suppose $\Gamma_0$ is a path-family of zero $p$-modulus containing all paths in $\Gamma(X)\setminus\Gamma_{u,g}$. Then for any $x_0\in X$ with $$x_0\notin \spt_p\Gamma_0\cup \{\M g^p=\infty \}$$ the map $u$ is continuous along any path $\gamma$ in the fundamental system of loops $\F_{x_0}(g;\Gamma_0)$.

In other words $\F_{x_0}(g;\Gamma_0)$ is a fundamental system of loops for $u$ for any $x_0\notin \spt_p\Gamma_0\cup \{\M g^p=\infty \}$.
\begin{proof}
Let $\gamma=\alpha\beta^{-1}$ where $\alpha,\beta\in \Gamma_{x_0x}\setminus \Gamma_0$ and $\M g^p(x)<\infty$. Since $\Gamma_0$ contains all the curves along which $g$ fails to be an upper gradient of $u$ it follows that $u$ is absolutely continuous along $\alpha$ and $\beta$, and consequently along $\gamma$.
\end{proof}
\end{lemma}

\bigskip\noindent\subsubsection*{Proof of Theorem \ref{liftchar}} With these tools we are ready to prove Theorem \ref{liftchar}. In fact we will give a more specific statement, from which the result follows.

\begin{theorem}\label{liftcond}
Let $1\le q<p$ be such that $X$ supports a weak $(1,q)$-Poincar\'e inqeuality.  Suppose that $u\in \Nem pXY$, $\rho = (\M g_u^q )^{1/q}\in L^p(\mu)$ and \[ N=\{ \M \rho^p=\infty \}\cup \spt_p(\Gamma(X)\setminus\Gamma_{u,\rho}). \] Let $x_0\in X\setminus N$ and $\hat y_0\in \phi^{-1}(u(x_0))$. Then $u$ has a lift $h\in \Nem pX{\widehat Y}$ with value $h(x_0)=\hat y_0$ in the strong sense if and only if, for some fundamental system of loops $\F_{x_0}(g;\Gamma_0)$ for $u$, we have $$u_\sharp\F_{x_0}(g;\Gamma_0)\le \phi_\sharp\pi(\widehat Y,\hat y_0).$$
\end{theorem}
\begin{proof}
Suppose $\F_{x_0}(g;\Gamma_0)$ is a fundamental system of loops for $u$ and  \[ u_\sharp(\F_{x_0}(g;\Gamma_0))\leq \phi_\sharp \pi(\hat Y,\hat y_0). \]  Note that by assumption $u(x_0)=\phi(\hat y_0)$. Moreover, $\Gamma_1=\Gamma_0\cup \Gamma(X)\setminus\Gamma_{u,\rho}$ is a path family of zero $p$-modulus for which $$\F_{x_0}(g;\Gamma_1)\subset \F_{x_0}(g;\Gamma_0). $$ Since $x_0\notin \spt_p(\Gamma_0)$ and $x_0\notin \spt_p(\Gamma(X)\setminus\Gamma_{u,\rho})$ it follows that $x_0\notin \spt_p(\Gamma_1)$ and $\F_{x_0}(g;\Gamma_1)$ is a fundamental system of loops for $u$. 

Let $\rho_0\ge 0$ be a $p$-integrable function such that $\displaystyle \int_\gamma\rho_0=\infty$ for all $\gamma\in \Gamma_1$ and $\M \rho_0^p(x_0)<\infty$. Note that $\rho_1:=g+\rho+\rho_0$ is a genuine upper gradient of $u$. Also notice that $\Gamma_1\subset \Gamma_{\rho_1}$.

Set $E=\{ \M\rho_1^p=\infty \}$ and define $h:X\setminus E\to \widehat Y$ by  \[ h(x):=\widehat{u\circ\alpha}(1)\textrm{ for any }\alpha\in \Gamma_{x_0x}\setminus\Gamma_1, \] where $\widehat{u\circ\alpha}$ is the lift of $u\circ\alpha$ starting at $\hat y_0$. Let us show that $h$ is well-defined.

Let  $\alpha,\beta\in \Gamma_{x_0x}\setminus\Gamma_1$. Then $\beta^{-1}\alpha\in \F_{x_0}(g; \Gamma_1)$ and by our assumption $[u\circ\beta^{-1}\alpha]\in \phi_\sharp\pi_1(\hat Y,\hat y_0)$. Thus the lift of $u\circ\beta^{-1}\alpha$ starting at $\hat y_0$ is a loop. In particular it follows that $\widehat{u\circ\alpha}(1)=\widehat{u\circ\beta}(1)$. Therefore $h(x)$ is well defined for $x\notin E$. We also have $\phi(h(x))=\phi(\widehat{u\circ\alpha}(1))=u\circ\alpha(1)=u(x)$ for $x\notin E$.

\bigskip\noindent Next, let us demonstrate that $h$ may be extended to a Newtonian map. For $x,y\in X\setminus E$ take paths $\alpha\in \Gamma_{x_0x}\setminus\Gamma_{\rho_1}$ and $\gamma\in \Gamma_{xy}\setminus\Gamma_{\rho_1}$ (both path families are non-empty by Lemma \ref{trivial}) whence $\gamma\alpha\in \Gamma_{x_0y}\setminus \Gamma_{\rho_1}$. Estimate \[ \hat d(h(x),h(y))=\hat d(\widehat{u\circ\gamma\alpha}(1),\widehat{u\circ\alpha}(1))\le \ell(\widehat{u\circ\gamma})=\ell(u\circ\gamma)\le \int_\gamma\rho_1. \] Taking infimum over all paths $\gamma\in \Gamma_{xy}\setminus\Gamma_{\rho_1}$ we obtain $$\hat d(h(x),h(y))\le f_{x,\rho_1}(y).$$ Since both $x$ and $y$ are Lebesgue points of $f_{x,\rho_1}$ we have, by \cite[Theorem 8.1.55]{HKST07} the inequality
\begin{align*}
&f_{x,\rho_1}(y)=\\
&|f_{x,\rho_1}(y)-f_{x,\rho_1}(x)|\le Cd(x,y)(\M \rho_1^q(x)^{1/q}+\M \rho_1^q(y)^{1/q}).
\end{align*}
Thus $h$ belongs to the \emph{Haj\l asz-Sobolev space} $M^{1,p}(X;\widehat Y)$ and by \cite[Theorem 10.5.3]{HKST07} a representative of $h$ is in $\Nem pX{\widehat Y}$. This is the desired lift of $u$. We abuse notation by denoting it again by $h$.

\bigskip\noindent For each $x\in X\setminus E$ we have $$ \hat d(\hat y_0,h(x))\le f_{x_0,\rho_1}(x) $$ whence by Lemma \ref{lebpoint} $$\lim_{r\to 0}\dashint_{B(x_0,r)}\hat d^p(\hat y_0,h)\ud\mu=0.$$ This concludes the proof of sufficiency.

\vspace{1cm}
\noindent Conversely, suppose $h\in \Nem pX{\hat Y}$ is a lift of $u\in\Nem pXY$ with \[\lim_{r\to 0}\dashint_{B(x_0,r)}\hat d^p(h,\hat y_0)\ud\mu=0. \] Let $\rho_0\ge 0$ be a $p$-integrable Borel function such that $$\int_\gamma\rho_0=\infty$$ for all $\gamma\in \Gamma(X)\setminus\Gamma_{u,\rho}$ and $\M \rho_0^p(x_0)<\infty$. Set $$\Gamma_0=\Gamma_{\rho+\rho_0}.$$ Notice again that $\rho+\rho_0$ is a genuine upper gradient of $u$ and in particular $\rho+\rho_0$ is an upper gradient for $u$ along any path $\gamma\in \Gamma(X)\setminus \Gamma_0$. By Lemma \ref{t0} the family $$\F:=\F_{x_0}(\rho+\rho_0;\Gamma_0)$$ is a fundamental system of loops for $u$ with basepoint $x_0$.

Let us show that $[u\circ\beta^{-1}\alpha]\in \phi_\sharp\pi(\widehat Y,\hat y_0)$ for any $\alpha,\beta\in \Gamma(X)\setminus \Gamma_{0}$ with $$\alpha(1)=\beta(1):=x\notin \{ \M(\rho+\rho_0)^p=\infty \}.$$ It suffices to show that the lift of the loop $u\circ\beta^{-1}\alpha$ to a path  in $\widehat Y$ starting at $\hat y_0$ is a loop.

Consider the paths $h\circ\alpha$ and $h\circ\beta$. Since $\int_\alpha\rho<\infty$ and $\int_\beta\rho<\infty$ it follows that $\rho(\alpha(t)),\rho(\beta(t))<\infty$ for almost every $t$, and by Lemma \ref{leb} $\alpha(t)$ and $\beta(t)$ are Lebesgue points both for $u$ and $h$, respectively, for a.e. $t$. By \cite[Theorem 8.1.55]{HKST07} we have \[ \hat d(h(\alpha(t)),h(\alpha(s)))\le Cd(\alpha(t),\alpha(s))[\rho(\alpha(t))+\rho(\alpha(s))]\textrm{ for a.e. }  t,s \in [0,1]\]  Consequently $h\circ\alpha$ agrees almost everywhere with an absolutely continuous path $f_\alpha$ joining the points $\hat y_0$ and $h(\alpha(1))=h(x)$. Similarly we have that $h\circ\beta$ agrees almost everywhere with an absolutely continuous path $f_\beta$ (joining $\hat y_0$ and $h(x)$).

On the other hand the equality $\phi\circ h(x)=u(x)$ holds for Lebesgue points $x$ of $h$. Therefore $\phi(h(\alpha(t)))=u(\alpha(t))$ and $\phi(h(\beta(t)))=u(\beta(t))$ for a.e. $t$. It follows that $\phi\circ h \circ \beta^{-1}\alpha=u\circ\beta^{-1}\alpha$ almost everywhere and consequently $f_\beta^{-1}f_\alpha$ is the lift of $u\circ\beta^{-1}\alpha$. However, $f_\beta^{-1}f_\alpha$ is a loop which is what we wanted to demonstrate.

\bigskip\noindent We have shown that $[u\circ\gamma]\in \phi_\sharp\pi(\widehat Y,\hat y_0)$ for every $\gamma\in \F$. Thus $$u_\sharp(\F)\le \phi_\sharp\pi(\widehat Y,\hat y_0)$$ and the claim is proven.
\end{proof}

\section{Stability}

For the purposes of this section we assume throughout that $X$ is a compact doubling $p$-Poincar\'e space and $Y$ a separable complete length space together with a covering map $\phi:\widehat Y\to Y$ that is a local isometry, unless otherwise stated in the claim.

\subsection{The subconjugacy property}

\begin{definition}\label{subconj}
Let $G$ be a group. A subgroup $H\subset G$ is said to satisfy the \emph{subconjugacy condition} if, for any sequence $(g_n)\subset G$, there exists some $g\in G$ so that the following condition is satisfied.
\begin{equation}\label{eq: subconj}
\liminf_{n\to \infty}H^{g_n}:=\bigcup_{k\ge 1}\bigcap_{n\ge k}H^{g_n}\le H^g.
\end{equation}
Here $H^g=g^{-1}Hg=\{g^{-1}hg: h\in H\}$.
\end{definition}

\noindent A subgroup satisfying the subconjugacy condition may also be referred to as \emph{subconjugate obedient}. Clearly any conjugate $H^g$ of a subconjugate obedient subgroup is again subconjugate obedient.

\begin{lemma}\label{noeth}
Any subgroup $H$ of a Noetherian group $G$ is \emph{subconjugate obedient}. More generally if the limit infimum of the sequence $H^{g_j}$, $$H_\infty:=\bigcup_{n\ge 1}\bigcap_{j\ge n}H^{g_j},$$ is finitely generated, then condition (\ref{eq: subconj}) is satisfied for some $g\in G$.
\begin{proof}
It suffices to prove the latter claim. Let $G,H$ and $(g_n)$ be as in Definition \ref{subconj} and let $S\subset H_\infty$ be a finite generating set for $H_\infty$. Then there exists $n_0$ so that $$h\in \bigcap_{j\ge n_0}H^{g_j}$$ for all $h\in S$. Thus $$S\subset \bigcap_{j\ge n_0}H^{g_j}\le H^{g_{n_0}}$$ and the claim follows.
\end{proof}
\end{lemma}

\begin{lemma}\label{hyptor}
Suppose $G$ is a torsion free hyperbolic group. Then $diag(G)$ has the subconjugacy property with respect to $G\times G$.
\begin{proof}
We shall denote by $C(S)$ the \emph{centralizer} of a set $S\subset G$ in $G$, i.e. $$C(S)=\{g\in G: gs=sg\textrm{ for all }s\in S \}.$$ If $S=G$ we write $C(S)=:Z(G)$ for the center $Z(G)$ of the group $G$.

Let $\{(g^1_j,g^2_j)\}_j\subset G\times G$ be a sequence and denote $H=\diag(G)$. Note that $$ H^{(g^1_j,g^2_j)}=H^{(1,(g^1_j)^{-1}g^2_j)}=:H^{(1,h_j)} $$ so we may consider sequences of the form $(1,h_j)$. For each $$H_n:=\bigcap_{j\ge n}H^{(1,h_j)}=\diag[C(h_ih_j^{-1}:i,j\ge n)]^{(1,h_n)}.$$ In fact, denoting $I_n=\{h_ih_j^{-1}:i,j\ge n \}$ we have $$\diag(C(I_n))^{(1,h_n)}=\diag(C(I_n))^{(1,h_k)}\textrm{ for any }k\ge n.$$ The sequence $H_n$ is ascending, as is the sequence $C_n=\diag(C(I_n))$.

If, for some $n_0$ the inclusion $I_{n_0}\subset Z(G)$ holds, then $C_{n_0}=\diag(G)=C_n$, $n\ge n_0$, and $$H_n=C_n^{(1,h_n)}=\diag(G)^{(1,h_{n_0})}\textrm{ for all }n\ge n_0,$$ implying $$H_\infty =\bigcup_{n\ge 1}H_n\le H^{(1,h_{n_0})}.$$

\bigskip\noindent We may therefore assume that for each $n$ there exists some $h_ih_j^{-1}\notin Z(G)$, $i,j>n$. We claim this implies that the groups $C_n$ are virtually cyclic. Indeed, since $h_ih_j^{-1}$ has infinite order the centralizer $C(h_ih_j^{-1})$ is virtually $\Z$ (Proposition \ref{zed}) and since $$C(I_n)\le C(h_ih_j^{-1})$$ the claim follows.

\bigskip\noindent Thus we have obtained an ascending sequence $C(I_n)\le C(I_{n+1})$ of virtually abelian groups which, by Lemma \ref{ACC} stabilizes. Let $n_0$ be such that $C_n=C_{n_0}$ for all $n\ge n_0$. Then also $$H_n=C_n^{(1,h_n)}=C_{n_0}^{(1,h_n)}=C_{n_0}^{(1,h_{n_0})}=H_{n_0},\ n\ge n_0.$$ This finishes the proof of the theorem.
\end{proof}
\end{lemma}
The proof crucially rests on the ACC for the subgroups $C(I_n)\le C(I_{n+1})$ and would work equally well for groups that satisfy the ACC for virtually abelian subgroups and for which, in addition, $C(g)$ is virtually abelian for any $g\notin Z(G)$. 

However the fundamental group of a nonpositively curved space need not have $C(g)$ virtually abelian as can easily be seen by considering the direct product of two free groups.

\subsection{Stability under convergence}

This subsection is devoted to the proof of Theorem \ref{stab}. We reduce Theorem \ref{stab} to the proof of the following statement. 
\begin{proposition}\label{stabi}
Suppose $\pi(Y)$ is countable and $\phi_\sharp\pi(\hat Y)\le \pi(Y)$ satisfies the subconjugacy condition. Let $(u_j)\subset \Nem pXY$ be a sequence with lifts $h_j\in \Nem pX{\hat Y}$ and \[ \sup_j \int_Xg_{u_j}^p\ud\mu<\infty. \] Suppose $(u_j)$ converges in $L^p(X;Y)$ to a map $u\in \Nem pXY$. Then $u$ admits a lift $h\in \Nem qX{\hat Y}$, where $1\le q<p$ is such that $(X,d,\mu)$ supports a weak $(p,q)$-Poincar\'e inequality.
\end{proposition}
\begin{proof}[Proof of Theorem \ref{stab}] Assume Proposition \ref{stabi}.
Let $u_j,h_j$ and $u$ be as in the claim and let $h\in \Nem qX{\hat Y}$ be a lift of $u\in \Nem pXY\subset \Nem qXY$ provided by Proposition \ref{stabi}. Since $u$ has a $p$-integrable upper gradient $g$, Lemma \ref{liftgrad} implies that $g$ is a $q$-weak upper gradient for $h$. It follows \cite[Theorem 8.1.55]{HKST07} that the pair $(h,g)$ satisfies the inequality \[ \hat d(h(x),h(y))\le Cd(x,y)[\M g^q(x)^{1/q}+\M g^q(y)^{1/q}] \] almost everywhere. The function $x\mapsto \M g^q(x)^{1/q}$ is $p$-integrable, so by \cite[Theorem 10.5.3]{HKST07} a representative of $h$ belongs to $\Nem pXY$. Clearly, this representative is a lift of $u$.
\end{proof}

\noindent Before going to the proof of Proposition \ref{stabi} let us present an auxiliary lemma.

\begin{lemma}\label{t7'}
Let $1\le q<p$. Suppose $u,u_j\in\Nem pXY$ for $j=1,2,\ldots$, and suppose $g_\infty,g_j\in L^p(\mu)$, $j=1,2,\ldots$, are nonnegative Borel functions. Let $\{\gamma^k\}_k$ be a countable collection of rectifiable paths in $X$ for which
\begin{itemize}
\item[(1)] $g_j$ is an upper gradient of $u_j$, and $g_\infty$ of $u$, along $\gamma^k$ for all $k$;
\item[(2)] $$\int_{\gamma^k}d(u,u_j)\to 0$$ as $j\to \infty$ for all $k$, and
\item[(3)] there are convex combinations $$\tilde g_j=\sum_{m\ge 0}\lambda^j_mg_m^{p/q}$$ so that $$\int_{\gamma^k}|\tilde g_j-g_\infty^{p/q}|\to 0$$ as $j\to \infty$ for all $k$.
\end{itemize}
Then there is a subsequence $(j_l)$ so that for all $k$ the paths $f_{j_l}^k=u_{j_l}\circ\gamma^k$ converge uniformly to $f^k=u\circ\gamma^k$ as $l\to \infty$.
\end{lemma}
\begin{proof}
For each $j$ and $k$ let $\tilde f^k _j:[0,1]\to Y$ be the constant speed parametrization of $f^k_j$, and $$ \ell^k_j(t)=\ell(f^k_j|_{[0,t]}), $$ so that $$ \tilde f^k_j(\ell^k_j(t)/\ell(f^k_j))=f^k_j(t). $$ Note that for each $j,k$ the path $\tilde f^k_k$ is $\ell(f^k_j)$-Lipschitz.

Set $$ a_k=\frac{2^{-k}}{1+M_k},\textrm{ where }M_k:=\sup_j\int_{\gamma^k}\tilde g_j <\infty $$ by (3). Using the same condition (3) we may estimate
\begin{align*}
\liminf_{j\to \infty} \sum_{k\ge 1} a_k\int_{\gamma^k}g_j^{p/q}&=\lim_{j\to \infty}\inf_{n\ge j}\sum_{k\ge 1} a_k\int_{\gamma^k}g_n^{p/q} \le \lim_{j\to \infty} \sum_{n\ge j}\lambda^j_n \sum_{k\ge 1} a_k\int_{\gamma^k}g_n^{p/q}\\
&=\lim_{j\to\infty} \sum_{k\ge 1} a_k\int_{\gamma^k}\tilde g_j=\sum_{k\ge 1} a_k\int_{\gamma^k}g_\infty^{p/q}<\infty.
\end{align*}
It follows that there is a subsequence (still denoted by the indices $j$) so that for each $k$ we have the uniform upper bound in $j$: \[ \int_{\gamma^k}g_j^{p/q}\le a_k^{-1}\sum_{m\ge 1} a_m\int_{\gamma^m}g_\infty^{p/q}=:C_k<\infty\textrm{ for all }j. \]

\bigskip\noindent Notice that \[ \ell(f^k_j)\le \int_{\gamma^k}g_j \le \ell(\gamma^k)^{1-q/p}\left( \int_{\gamma^k}g_j^{p/q} \right)^{q/p}\le \ell(\gamma^k)^{1-q/p}C_k^{q/p}.  \] Thus the maps $\tilde f^k_j$ are uniformly Lipschitz. Using this and a diagonal argument we may pass to a subsequence with $\tilde f^k_j$ converging uniformly to some map $\tilde f^k$ as $j\to \infty$, for all $k$.

Notice also that
\begin{align*}
|\ell^k_j(b)-\ell^k_j(a)|=\ell(f^k_j|_{[a,b]})&\le \ell(\gamma^k)\int_a^bg_j\circ\gamma^k \ud \tau\\
&\le (\ell(\gamma^k)|b-a|)^{1-q/p}\left( \int_{\gamma^k}g_j^{p/q} \right)^{q/p}
\end{align*}
for all $a,b\in [0,1]$. Thus the family $(\ell^k_j)_j$ is equicontinuous and we may again use the Arzela-Ascoli Theorem \cite[Theorem 10.28]{hei01} to pass to a subsequence for which $\ell^k_j$ converges uniformly to a function $\ell^k$ as $j\to \infty$, for all $k$.

\bigskip\noindent We have now passed to a subsequence, which we denote henceforth with the indices $j_l$ for which $\tilde f^k_{j_l}$ and $\ell^k_{j_l}$ both converge uniformly to $\tilde f^k$ and $\ell^k$ as $l\to \infty$, for all $k$.

Suppose $\ell^k\ne 0$. Then $$f_{j_l}^k=\tilde f^k_{j_l}\circ(\ell^k_{j_l}/\ell(f^k_{j_l}))\to \tilde f^k\circ(\ell^k/\ell^k(1))$$ uniformly as $l\to \infty$. On the other hand by (2) the sequence $f_j^k$ has $L^1([0,1])$-limit $f^k$ so $$ f^k=\tilde f^k\circ(\ell^k/\ell^k(1))\textrm{ a.e.}.$$ Since both sides are continuous we conclude that equality holds everywhere.

If $\ell^k\equiv 0$ then $\ell(f)\le \liminf_{l\to\infty}\ell^{k}_{j_l}(1)=0$ and in fact $\tilde f^k_{j_l}\to const.=c_k$ uniformly as $l\to \infty$, from which it follows that $f^k_{j_l}=\tilde f^k_{j_l}\circ(\ell^k_{j_l}/\ell(f^k_{j_l}))\to c_k$ uniformly. Again by (2) we have $f^k\equiv c_k$ for each $k$.

\end{proof}

\begin{proof}[Proof of Proposition \ref{stabi}]
We proceed in three steps.
\subsubsection*{Step 1} Choosing a fundamental system of loops for $u$. 

\bigskip\noindent The sequence $\displaystyle (g_{u_j}^{p/q})_j$ is bounded in $L^q(\mu)$. We may pass to a subsequence so that 
\begin{itemize}
\item[(1)] $u_j\to u$ pointwise almost everywhere,
\item[(2)] $g_{u_j}^{p/q}$ converges weakly in $L^q(\mu)$ to a function $g=g_\infty^{p/q}\in L^q(\mu)$ and a convex combination $\displaystyle \tilde g_j= \sum_{n\ge j}\lambda_n^jg_{u_n}^{p/q}$ converges to $g$ in $L^q$-norm,
\item[(4)] $\int_\gamma d(u_j,u)\to 0$ as $j\to \infty$ for $q$-almost every curve $\gamma$ in $X$, and 
\item[(5)] $\displaystyle \lim_{j\to \infty}\int_\gamma |\tilde g_j-g| =0$ for $q$-almost every curve $\gamma$ in $X$.
\end{itemize}
The function $g_\infty$ is a $q$-weak upper gradient for $u$, see \cite[Lemma 3.1]{kal01}.

Denote by $\Gamma_0'$ the $q$-exceptional path family to (4) and (5). Let
\begin{align*}
\rho_j&=(\M g_{u_j}^q)^{1/q}, N_j=\{\M \rho_{j}^q=\infty \},\\
\rho_\infty &=(\M g_\infty^q)^{1/q}\textrm{ and } N_\infty=\{ \M\rho_\infty^q=\infty\}.
\end{align*}

Set $$N':=(N_\infty\cup \spt_q(\Gamma(X)\setminus \Gamma_{u,\rho_\infty}))\cup\bigcup_{j=1}^\infty (N_j\cup \spt_q(\Gamma(X)\setminus \Gamma_{u_j,\rho_j}))$$ and $$N=\spt_q(\Gamma_0')\cup\spt_q(\Gamma(X)\setminus \Gamma_{u,\rho_\infty})\cup \{\M \rho_\infty^p=\infty\}\cup N'.$$ Choose a point $x_0\notin N$ so that $u_j(x_0)\to u(x_0)$ as $j\to\infty$. By Theorem \ref{liftcond} there are fundamental systems of loops $\F_{x_0}(g_j,\Gamma_j)$ for $u_j$ so that \[ (u_j)_\sharp\F_{x_0}(g_j;\Gamma_j)\le \phi_\sharp\pi(\widehat Y,h_j(x_0)). \]



Set \[\Gamma_0=(\Gamma(X)\setminus\Gamma_{u,\rho_\infty})\cup\Gamma_0'\cup\bigcup_{j\ge 1}\Gamma_j \] and \[ g_0=g_\infty+\sum_{j=1}^\infty 2^{-j}g_j. \] Notice that since $x_0\notin \spt_q\Gamma_0'$ and $x_0\notin \spt_q\Gamma_j$ for all $j\ge 1$, it follows that $x_0\notin \spt_q\Gamma_0$. By Lemma \ref{t0} the path family $\F_{x_0}(g_0; \Gamma_0)$ is a fundamental system of loops for $u$, with basepoint $x_0$.

Moreover $$\F_{x_0}(g_0; \Gamma_0)\subset \F_{x_0}(g_j,\Gamma_j)$$ for all $j$, since $\Gamma_j\subset \Gamma_0$ for all $j$ and $\M g_0^q(x)<\infty$ implies $\M g_j^q(x)<\infty$ for all $j$.

\bigskip\noindent To prove Proposition \ref{stabi} it suffices to prove that
\begin{equation}\label{claim}
u_\sharp\F_{x_0}(g_0;\Gamma_0)\le \phi_\sharp\pi_1(\hat Y,\hat y_0)
\end{equation}
for some $\hat y_0\in \phi^{-1}(u(x_0))$.

We will split the inclusion (\ref{claim}) into two parts.

\subsubsection*{Step 2} A first inclusion for the induced image subgroup of the chosen fundamental system of loops.
\begin{claim} There is a subsequence for which
\begin{equation}\label{step2}
u_\sharp\F_{x_0}(g_0;\Gamma_0)\le \bigcup_{n\ge 0}\bigcap_{j\ge n}\phi_\sharp\pi(\widehat Y,h_j(x_0)).
\end{equation}
\end{claim}

\bigskip\noindent Indeed, since $\pi_1(Y,u(x_0))$ is countable, the subgroup $$u_\sharp\F_{x_0}(g_0,\Gamma_0)\le \pi_1(Y,u(x_0))$$ is generated by a countable number of loops $u\circ\gamma^1,u\circ\gamma^2,\ldots$ where $\gamma^1,\gamma^2,\ldots \in \F_{x_0}(g_0,\Gamma_0)$; that is, \[ u_\sharp\F_{x_0}(g_0,\Gamma_0) = \langle u\circ\gamma^k: k\in \N\rangle .  \] By Lemma \ref{t7'} we may pass to a subsequence (still denoted by the same indices) so that $$u_j\circ\gamma^k\to u\circ\gamma^k\textrm{ as }j\to\infty$$ uniformly for all $k$. Let us denote (as in Lemma \ref{t7'}) $f^k_j=u_j\circ\gamma^k$ and $f^k=u\circ\gamma^k$.

\bigskip\noindent For large enough $j$ we have $d(u_{j}(x_0),u(x_0))<r_{u(x_0)}$ so there is a path $\alpha_{j}$ joining $u_{j}(x_0)$ to $u(x_0)$ in $B(u(x_0),r_{u(x_0)})$. Lifting this to a path in $\hat Y$ starting at $h_{j}(x_0)$ we obtain a path $\hat \alpha_{j}$ joining $h_{j}(x_0)$ to a point $\hat \alpha_{j}(1)=:\hat q_{j}\in \phi^{-1}(u(x_0))$ lying in the same component of $p^{-1}B(u(x_0),r_{u(x_0)})$ as $h_j(x_0)$.

Let $k$ be arbitrary but fixed. By  the uniform convergence, for large enough $j$ the distance $$\max_{0\le t\le 1} d(f_j^k(t),f^k(t))$$ is small enough  to $f^k$ so that $f^k\simeq \alpha_{j}^{-1}f^k_{j}\alpha_{j}$, $k=1, 2, \ldots$. Thus, for large enough $j$ we have
\begin{align*}
[u\circ\gamma^k]\in &[\alpha_{j}]^{-1}(u_{j})_\sharp\F_{x_0}(g_j,\Gamma_j)[\alpha_{j}]\le [\phi\circ\widehat\alpha_{j}]^{-1}\phi_\sharp\pi(\widehat{Y},h_j(x_0))[\phi\circ\widehat\alpha_{j}]\\
=&\phi_\sharp\pi(\widehat Y,\hat q_{j}).
\end{align*}
Inclusion (\ref{step2}) follows.

\subsubsection*{Step 3} Using the subconjugacy property.
\begin{claim} The following inclusion holds true for some $\hat y_0\in \phi^{-1}(u(x_0))$.
\begin{equation}\label{step3}
H_\infty:=\bigcup_{n\ge 1}\bigcap_{j\ge n}\phi_\sharp\pi(\hat Y,\hat q_{j} )\le \phi_\sharp\pi(\hat Y,\hat y_0)
\end{equation}
\end{claim}

\bigskip\noindent Indeed, each $\phi_\sharp\pi(\hat Y,\hat q_{j})$ is a conjugate of $\phi_\sharp\pi(\hat Y,\hat q_0)$ for some (arbitrary) fixed $\hat q_0\in \phi^{-1}(u(x_0))$. Denote by $a_j\in \pi(Y,u(x_0))$ the conjugating elements for $\phi_\sharp\pi(\hat Y,\hat q_{j})$. Since we assume that $\phi_\sharp\pi(\hat Y,\hat q_0)$ has the subconjugacy property it follows that there is some $a\in \pi(Y,u(x_0)))$ so that $$ H_\infty \le \phi_\sharp\pi(\hat Y,\hat q_0)^a.$$ Now $\phi_\sharp\pi(\hat Y,\hat q_0)^a=\phi_\sharp\pi(\hat Y,\hat y_0)$ for some $\hat y_0\in \phi^{-1}(u(x_0))$.

\bigskip\noindent We have proven (\ref{claim}), whereby Theorem \ref{liftcond} implies that $u$ admits a lift $h\in \Nem qX{\hat Y}$. This concludes the proof of Proposition \ref{stabi}.
\end{proof}

\noindent In general the subconjugacy property is not valid for arbitrary subgroups of a finitely generated group.
To rid ourselves of the assumption of subconjugacy it would be sufficient to prove that the subgroup $u_\sharp\F_{x_0}(g_0,\Gamma_0)$ is finitely generated (this follows essentially from the proof of \ref{noeth}). This sort of result requires some topological properties on the domain space $(X,d,\mu)$, such as having a finitely generated fundamental group. Up to now we have not imposed any such requirements. For instance assuming that $X$ admits a universal covering would be a natural requirement in this setting.

The question whether $u_\sharp\F_{x_0}(g_0,\Gamma_0)$ is finitely generated is related to the question of whether a map $u\in \Nem pXY$ induces a homomorphism $$u_*:\pi(X)\to \pi(Y)$$ since of course such a homomorphism would transfer any finite generating set for $\pi(X)$ (assuming it has one) to a finite generating set for $u_*(\pi(X))$.

\begin{corollary} Assume $X$ has finitely generated fundamental group, $u\in \Nem pXY$ has $p$-weak upper gradient $g\in L^p(\mu)$ and $x_0\in X$ satisfies $\M g^p(x_0)<\infty$. Suppose there is a homomorphism $$u_*:\pi(X,x_0)\to \pi(Y,u(x_0))$$ and a fundamental system of loops for $u$, $\F_{x_0}(g_0;\Gamma_0)$, so that 
\begin{itemize}
\item[(1)] for all $[\gamma]\in \pi(X,x_0)$ there exists $\gamma'\in \F_{x_0}(g_0;\Gamma_0)$ satisfying \[ u_*([\gamma])=[u\circ\gamma'], \] and
\item[(2)] \[ u_*([\gamma])=[u\circ\gamma] \] for all $\gamma\in \F_{x_0}(g_0;\Gamma_0)$.
\end{itemize}
Then $$u_\sharp\F_{x_0}(g_0,\Gamma_0)$$ is finitely generated.
\end{corollary}
\begin{proof}
By (1) we have $$u_*\pi(X,x_0)\le u_\sharp\F_{x_0}(g_0;\Gamma_0)$$ and (2) implies $$ u_\sharp\F_{x_0}(g_0;\Gamma_0)\le u_*\pi(X,x_0).$$ The group $\pi(X,x_0)$ is finitely generated. Thus $u_\sharp\F_{x_0}(g_0;\Gamma_0)$ is finitely generated.
\end{proof}

\subsection{Existence of minimizers}

\begin{proposition}\label{comp}
Let $X$ be a compact doubling $p$-Poincar\'e space and $Y$ a compact geodesic path representable space whose fundamental group is Noetherian or torsion free hyperbolic. Then any $p$-quasihomotopy class $[v]\subset \Nem pXY$ is closed under bounded convergence, i.e. if $u_j\in [v]$ is a sequence with $$\sup_j\int_Xg_{u_j}^p\ud\mu<\infty$$ and $u_j\to u\in \Nem pXY$ in $L^p(\mu)$, then $u\in [v]$.
\end{proposition}

\begin{proof}[Proof of Proposition \ref{comp}]
Recall the covering $p: \diaco Y\to Y\times Y$ from Subsection 2.4. Note that if $G=\pi(Y)$ satisfies any of the conditions in the claim of Theorem \ref{minex} then $diag(G)=p_\sharp\pi(\diaco Y)$ has the subconjugacy property with respect to $\pi(Y\times Y)=G\times G$ (Lemmas \ref{noeth} and \ref{hyptor}).

Suppose the sequence $u_j$ is as in the claim and let $h_j\in \Nem pX{\diaco Y}$ be the lift of $(v,u_j)\in \Nem pX{Y\times Y}$ the existence of which is guaranteed by Theorem \ref{homchar}. The norms of the gradients \[ \int_Xg_{(v,u_j)}^p\ud\mu\le C \int_Xg_u^p\ud\mu+C\int_Xg_{v_j}^p\ud\mu \] are bounded uniformly in $j$ and the sequence $(v,u_j)$ converges to the map $(u,v)\in \Nem pX{Y\times Y}$ in $L^p(X;Y\times Y)$. By Theorem \ref{stab} $(u,v)$ has a lift $h\in \Nem pX{\diaco Y}$ and thus (again by Theorem \ref{homchar}) $v\in [u]$.
\end{proof}

\noindent Theorem \ref{minex} readily follows from Proposition \ref{comp} and the lower semicontinuity of $e_p$ under $L^p$-convergence (\cite[Theorem 7.3.9]{HKST07}).
\begin{proof}[Proof of Theorem \ref{minex}]
Given a $p$-quasihomotopy class $[v]\subset \Nem pXY$ take a minimizing sequence $u_j\in [v]$, $$ e_p(u_j)\to \inf_{u\in [v]}e_p(u)=:I.$$ Since $u_j$ is a bounded sequence in $\Nem pXY$ we may, by Rellich's compactness theorem \ref{rellich}, pass to a subsequence converging in $L^p(X;Y)$ to a map $u\in \Nem pXY$ so that $$e_p(u)=\int_X g_u^p\ud \mu\le \liminf_{j\to\infty}\int_Xg_{u_j}^p\ud\mu=I.$$ By Proposition \ref{comp} we have $u\in [v]$ and therefore $I\le e_p(u)\le I$. This completes the proof.
\end{proof}

\bigskip\subsubsection*{Acknowledgements} I would like to thank Pekka Pankka for a careful reading of the manuscript and many helpful suggestions. The manuscript was finished during the research term on analysis and geometry in metric spaces in ICMAT, Madrid; many thanks go to the organizers of this event for a hospitable and inspiring stay. I am also indebted to my advisor Ilkka Holopainen and the V\"ais\"al\"a foundation for financial support.

\bibliographystyle{plain}
\bibliography{abib}
\end{document}